\documentclass[11pt]{amsart}

\usepackage{amsmath,amscd,amssymb,amsfonts,amsthm,hyperref}
\usepackage{xcolor}  
\hypersetup{   colorlinks,   linkcolor={red!50!black},    citecolor={blue!70!black},   urlcolor={blue!80!black}}

\usepackage{tcolorbox}
\usepackage[cmtip,matrix,arrow]{xy}
\usepackage{tikz}
\usepackage{tikz-cd}

\newtheorem{theorem}{Theorem}[section]
\newtheorem{corollary}[theorem]{Corollary}
\newtheorem{proposition}[theorem]{Proposition}

\newtheorem{lemma}[theorem]{Lemma}
\theoremstyle{definition}    
\newtheorem{definition}[theorem]{Definition}
\theoremstyle{remark}

\newtheorem{remark}[theorem]{Remark}

\newtheorem{example}[theorem]{Example}
\newcommand\A{\mathcal{A}}

\newcommand\M{\mathcal{M}}
\newcommand\G{\mathcal{G}}

\renewcommand{\L}{\mathcal{L}}

\newcommand{\T}{\mathcal{T}}

\newcommand{\ca}{\mathcal}

\newcommand{\N}{\mathbb{N}}
\newcommand{\R}{\mathbb{R}}

\newcommand{\Z}{\mathbb{Z}}

\newcommand\pt{\on{pt}}
\renewcommand{\P}{\ca{P}}
\newcommand{\FF}{{\mathsf{F}}}

\newcommand\lie[1]{\mathfrak{#1}}
\renewcommand{\k}{\lie{k}}

\newcommand{\g}{\lie{g}}

\renewcommand{\a}{\mathsf{a}}

\newcommand{\on}{\operatorname}

\newcommand{\Aut}{ \on{Aut} } 
\newcommand{\Gau}{ \on{Gau} }
  
 \newcommand{\gau}{ \mf{gau} }
\newcommand{\Ad}{ \on{Ad} }

\newcommand{\Hom}{ \on{Hom}}

\renewcommand{\ker}{ \on{ker}}

\newcommand{\SO}{ \on{SO}}
\newcommand{\Mult}{  \on{Mult}}

\newcommand\qu{/\kern-.7ex/} 

\newcommand{\lra}{\longrightarrow}
\newcommand{\hra}{\hookrightarrow}

\renewcommand{\d}{{\mathsf{d}}}

\newcommand{\f}{\frac}

\newcommand{\p}{\partial}
\renewcommand{\l}{\langle}
\renewcommand{\r}{\rangle}
\newcommand\hh{{\f{1}{2}}}

\newcommand{\ti}{\tilde}
\newcommand{\eeq}{\end{eqnarray*}}
\newcommand{\beq}{\begin{eqnarray*}}

\newcommand{\D}{\ca{D}}

\newcommand{\pr}{\on{pr}}
\newcommand{\wh}{\widehat}
\newcommand{\wt}{\widetilde}
\newcommand{\mf}{\mathfrak}
\newcommand{\rra}{\rightrightarrows}

\newcommand{\Vect}{\on{Vect}}
\newcommand{\Diff}{\on{Diff}}
\newcommand{\PSL}{\on{PSL}}

\newcommand{\SL}{\on{SL}}
\newcommand{\RP}{\R\!\on{P}}
\renewcommand{\S}{\ca{S}}
\newcommand{\vir}{\mf{vir}}

\newcommand{\CC}{\mathsf{C}}
\newcommand{\PP}{\mathsf{P}}
\newcommand{\ignore}[1]{}
\newcommand{\sz}{\mathsf{s}}
\newcommand{\tz}{\mathsf{t}}
\renewcommand{\subset}{\subseteq}
\newcommand{\DC}{\ca{D}(\CC)}
\newcommand{\op}{{\on{op}}}

\begin{document}
\sloppy

\title{On the coadjoint Virasoro action}
\author{Anton Alekseev}
\author{Eckhard Meinrenken}

\begin{abstract}	
The set of coadjoint orbits of the Virasoro algebra at level 1 is in bijection with the set of conjugacy classes in a certain open subset $\wt{\SL}(2,\R)_+$ of the universal cover of $\SL(2,\R)$. We strengthen this bijection to a Morita equivalence of quasi-symplectic groupoids, integrating the Poisson structure on $\vir^*_\mathsf{1}(S^1)$ and the Cartan-Dirac structure on  $\wt{\SL}(2,\R)_+$, respectively.
\end{abstract}

\maketitle
\tableofcontents

\section{Introduction}
Let $\CC$ be an unparametrized circle: a compact, connected, oriented 1-manifold. The Virasoro Lie algebra $\vir(\CC)$ is the canonical central extension of the Lie algebra $\Vect(\CC)$ of vector fields. Its smooth dual at level $1$, denoted $\vir^*_1(\CC)$, has a geometric interpretation as the space of projective structures on $\CC$, and also as a space of Hill operators on $\CC$. 
The group $\Diff_+(\CC)$ of orientation preserving diffeomorphisms acts on   $\vir^*_1(\CC)$  by affine-linear transformations;
we shall refer to this action as the  \emph{coadjoint Virasoro action}.   
   
 A classification of the coadjoint Virasoro orbits was achieved in  the work of  Kirillov \cite{kir:orb},  Lazutkin-Pankratova \cite{laz:nor}, Goldman \cite{gol:the}, Segal \cite{seg:uni}, and Witten \cite{wit:coa}. In the approach of \cite{gol:the,seg:uni}, this classification is described in terms of a diagram 
\begin{equation}\label{eq:corr0}
\xymatrix{ & \DC\ar[dl]\ar[dr] & \\
	\vir^*_1(\CC) & & \wt{\SL}(2,\R)_+
}\end{equation}
Here $\DC$ is the space of \emph{developing maps} for projective structures on $\CC$. A developing map is an orientation-preserving immersion
\[ \gamma\colon \wt{\CC}\to \RP(1)\]
of the simply connected cover of $\CC$, such that $\gamma$ is \emph{quasi-periodic}, with monodromy given by the action of an element of $\PSL(2,\R)$. The right arrow takes $\gamma$ to its \emph{lifted} monodromy in the universal cover, and has as its image a certain open subset indicated by a subscript `$+$'. The left arrow takes $\gamma$ to the corresponding projective structure. The space $\DC$ has natural commuting actions of $\PSL(2,\R)$ and of the universal cover of $\Diff_+(\CC)$, and the two arrows are the quotient maps for these two actions.  Consequently,  the diagram sets up a bijection between the  coadjoint orbits in $\vir^*_1(\CC)$ and conjugacy classes in $\wt{\SL}(2,\R)_+$. 

In this article, we show that this correspondence of group actions extends canonically to a correspondence between the Poisson structure on $\vir^*_1(\CC)$ and the Cartan-Dirac structure on $\wt{\SL}(2,\R)_+$. More precisely, we construct a Morita equivalence, in the sense of Xu \cite{xu:mom}: 
\begin{equation}\label{eq:corr02}
\xymatrix{ (\G_1,\omega_1) \ar[d]<2pt>\ar@<-2pt>[d]& (\DC,\varpi_\D)\ar[dl]\ar[dr] & (\G_2,\omega_2) \ar[d]<2pt>\ar@<-2pt>[d]\\
	\vir^*_1(\CC) & & \wt{\SL}(2,\R)_+
}\end{equation}
In this diagram, $(\G_1,\omega_1)$ is a symplectic groupoid integrating the Poisson structure on $\vir^*_1(\CC)$, and 
$(\G_2,\omega_2)$ is a quasi-symplectic groupoid integrating the Cartan-Dirac structure on $\wt{\SL}(2,\R)_+$. 
The groupoid  $\G_2$ is the action groupoid for the conjugation action of $\PSL(2,\R)$. On the other hand,  $\G_1$ is \emph{not} simply an action 
groupoid  -- rather, it is obtained as a \emph{quotient} of an action groupoid.   The Morita equivalence between these quasi-symplectic groupoids has the space of developing maps as its Hilsum-Skandalis bimodule, and comes with a distinguished bi-invariant  
2-form 
\[ \varpi_\D\in \Omega^2(\DC).\]
For $\CC=S^1$, writing quasi-periodic paths as $\gamma=(\sin\phi:\cos\phi)$
with $\phi\colon \R\to \R$, this 2-form is given by an expression
\begin{equation}\label{eq:theformula}
 \varpi_\D=\int_0^1 \d \Big((\phi')^2+\hh \S(\phi) \Big)\wedge \ \f{\d\phi}{\phi'}+\mbox{boundary terms}\end{equation}
involving the Schwarzian derivative $\S(\phi)$; the boundary terms are needed to make the expression invariant under diffeomorphisms.

The Morita equivalence \eqref{eq:corr02} of quasi-symplectic groupoids gives rise to a 1-1 correspondence of the associated 
Hamiltonian spaces. It hence allows us to associate (certain) Hamiltonian Virasoro spaces to finite-dimensional $\PSL(2,\R)$-spaces 
with $\wt{\SL}(2,\R)_+$-valued moment maps and vice versa. One such example, motivated by the  recent physics literature on Jackiw-Teitelboim gravity (e.g., \cite{sa:jt,sta:jt}), is the Teichm\"uller moduli space of conformally compact hyperbolic metrics on surfaces with boundary. Details of these and similar examples will be discussed in forthcoming work.  

The Morita equivalence for the Virasoro Lie algebra is related to a similar 
Morita equivalence for loop groups, via Drinfeld-Sokolov reduction. For any connected Lie group $G$ with an invariant metric on its Lie algebra $\g$, the 
loop algebra $L\g$ has a central extension $\wh{L\g}$. Taking $G=\PSL(2,\R)$, the Drinfeld-Sokolov procedure \cite{dri:kdv} realizes $\vir^*_1(S^1)$ as a Marsden-Weinstein reduced space of $\wh{L\g}^*_1$; we shall work with the coordinate-free description of this  method, as described by Segal \cite{seg:geo}. An analogue of the 2-form $\varpi_\D$ for the loop group setting (for more general $G$) was given in our earlier work \cite{al:mom,al:ati}, and used to establish a 1-1 correspondence between Hamiltonian loop group spaces and finite-dimensional q-Hamiltonian  spaces. See \cite{xu:mom} for its interpretation as a Morita equivalence.

Throughout this article, we will treat  infinite-dimensional manifolds as Fr\'{e}chet manifolds \cite{ham:inv}, but without entering any technical discussion. Recent references giving a detailed treatment of the relevant techniques in a related context are \cite{die:gro,rie:ham}.
 Alternatively, one could work with diffeologies \cite{igl:dif}. Recall that diffeology is particularly well-suited for dealing with differential forms, which actually suggests diffeological symplectic and quasi-symplectic groupoids as  convenient stand-ins for infinite-dimensional Poisson and Dirac manifolds.  

The organization of the paper is as follows. Section \ref{sec:vir} discusses Hill operators and the Virasoro algebra. The approach will be coordinate-free; the coordinate expressions are spelled out towards the end of the section. Section \ref{sec:dev} studies the geometry of the space of developing maps. In particular, we determine the stabilizers for the actions of diffeomorphisms; this leads us to the construction of the groupoid $\G_1$. Section \ref{sec:mor} describes the 2-form $\varpi_\D$ 
on the space of developing maps, giving the Morita equivalence of quasi-symplectic groupoids and resulting in a correspondence of their Hamiltonian spaces. Section \ref{sec:ds} gives the  construction of $\varpi_\D$ by Drinfeld-Sokolov reduction, which accounts for its basic properties. The appendix gives a general review of Morita equivalence of quasi-symplectic Lie groupoids, as defined by Xu \cite{xu:mom}.  
\bigskip

{\bf Acknowledgments.} We are grateful to Henrique Bursztyn, Rui Fernandes, and Yiannis Loizides for helpful
discussions on various aspects of this work. Research of A.A. was supported in part by the grants 182767, 208235, 200400 and by the NCCR SwissMAP of the Swiss National Science Foundation (SNSF), and by the award of
the Simons Foundation to the Hamilton Mathematics Institute of the Trinity College
Dublin under the program \emph{Targeted Grants to Institutes}. Research of E.M. was supported by an NSERC Discovery Grant.

\section{Virasoro Lie algebra}\label{sec:vir}
We begin by reviewing some background on the Virasoro Lie algebra, and its relationship with Hill operators.  We shall follow the coordinate free approach from \cite{fre:ver,seg:uni}; further information may be found in \cite{khe:inf,kir:orb,ovs:pro}.

\subsection{Hill operators}
Let $\CC$ be a 1-dimensional oriented manifold. For $r\in \R$, we denote by 
$|\Lambda|^r_\CC\to \CC$ the $r$-density bundle on $\CC$; its space of sections is denoted by 
\[ |\Omega|^r_\CC=\Gamma(|\Lambda|^r_\CC).\]
The orientation on $\CC$ identifies $|\Lambda|^r_\CC$  with the cotangent bundle for $r=1$, and with the 
tangent bundle for $r=-1$. 
The principal symbol of a $k$-th order differential operator 
\[ D\colon |\Omega|^{r_1}_\CC\to  |\Omega|^{r_2}_\CC\]
is an element 
\[ \sigma_k(D)\in \Gamma\big(\on{Sym}^k(T\CC)\otimes \Hom(|\Lambda|^{r_1}_\CC,|\Lambda|^{r_2}_\CC)\big)
\cong  |\Omega|^{r_2-r_1-k}_\CC; 
\]
the principal symbol of a composition of two such operators is the product of the principal symbols. 
The formal  adjoint 
of $D$ is the $k$-th order differential operator 
\[ D^*\colon |\Omega|_\CC^{1-r_2}\to |\Omega|_\CC^{1-r_1}\]
given by $\int_\CC (Du)v=\int_\CC u(D^*v)$ whenever $u$ has compact support; its principal symbol is 
$\sigma_k(D^*)=(-1)^k \sigma_k(D)$. 

\begin{example}\phantom{.}
	\begin{enumerate}
		\item 
		Every $v\in \Vect(\CC)$ defines a first order differential operator $\L_v\colon |\Omega|^{r_1}_\CC\to  |\Omega|^{r_1}_\CC
		$ given by {Lie derivative}, with $\sigma_1(\L_v)=v$. The formal adjoint is 
		$\L_v^*=-\L_v$ as an operator $|\Omega|^{1-r_1}_\CC\to  |\Omega|^{1-r_1}_\CC$. 
		\item
		Every  $u\in  |\Omega|^{r}_\CC$ defines a $0$-th order differential operator $M_u\colon  |\Omega|^{r_1}_\CC\to  |\Omega|^{r_1+r}_\CC$ given by multiplication, with $\sigma_0(M_u)=u$. The formal adjoint  is 
		 $M_u^*=M_u\colon |\Omega|^{1-r_1-r}_\CC\to  |\Omega|^{1-r_1}_\CC$. 
	\end{enumerate}
\end{example}
A \emph{Hill operator} is a second order differential operator 
\[ L\colon  |\Omega|^{ -\f{1}{2}}_\CC \to  |\Omega|^{ \f{3}{2}}_\CC 
\]
with $L^*=L$ and with principal symbol $\sigma_2(L)=1$.  (The choices $r_1=-\f{1}{2},\ r_2=\f{3}{2}$ are the unique ones for which these conditions make sense.) We denote by $\on{Hill}(\CC)$ the affine space  of all Hill operators; the underlying linear space is the space $|\Omega|^2_\CC$ of
quadratic differentials. 

It is often convenient to choose a connection on the bundle $|\Lambda|^{-\f{1}{2}}_\CC$, or equivalently a  first order differential operator 
\begin{equation}\label{eq:partial}\partial\colon  |\Omega|^{ -\f{1}{2}}_\CC \to  |\Omega|^{ \f{1}{2}}_\CC\end{equation}
with $\sigma_1(\partial)=1$. (Any two such differ by $M_u$ for some 
$u\in |\Omega|^1_\CC$.)  
Then $\partial^*\colon  |\Omega|^{\f{1}{2}}_\CC \to  |\Omega|^{ \f{3}{2}}_\CC $, and the 
	Hill operators on $\CC$ are of the 
	form 
	\[ L=-\partial^*\circ \partial+\T,\] 
	where $\T\in |\Omega|^{2}_\CC$ is a quadratic differential.  
Using \eqref{eq:partial} we may define the \emph{Wronskian} of two $-\hh$-densities as 
\begin{equation}\label{eq:wronskian}
 W(u_1,u_2)=u_1\partial u_2-u_2\partial u_1\in |\Omega|^{0}_\CC;\end{equation}
this does not depend on the choice of $\partial$.

The group $\Diff_+(\CC)$ of orientation preserving diffeomorphisms acts on the spaces of $r$-densities by push-forward. 
This induces an affine action on $\on{Hill}(\CC)$ by 
\begin{equation}\label{eq:hillaction} \FF\cdot L=\FF_*\circ L\circ (\FF^{-1})_*,\end{equation}
with underlying linear action the push-forward  of
quadratic differentials. Given a Hill operator $L$ and any $v\in \Vect(\CC)$, the Lie derivative $\L_v(L)=[\L_v,L]$ is a differential operator of order $0$, and hence is multiplication by a  quadratic differential. Denoting the latter by $-D_L v$, this defines a linear map
\begin{equation}\label{eq:dv} D_L\colon |\Omega|^{-1}_\CC\to  |\Omega|^{2}_\CC,\end{equation}
satisfying $D_L [v_1,v_2]=\L_{v_1} (D_L v_2)-\L_{v_2}(D_L v_1)$. By definition, 
\begin{equation}\label{eq:infhillaction} v\cdot L=-D_L v\end{equation}
describes the infinitesimal action on the space of Hill operators. From the coordinate description below, we
will see below that $D_L$ is a third order differential operator, with 
\[ D_L^*=-D_L,\ \ \sigma_3(D_L)=-\hh.\]
Finally, we note that the Hill operator determines for all $x_0\in \CC$,  a symmetric bilinear form (`bilinear concomitant')
on the 2-jet bundle $J^2(T\CC)$ 
\[ B_{L}\colon J^2(T\CC)\times_{\CC} J^2(T\CC)\to \R.\]
At $x_0\in \CC$, this bilinear form is given by 
\begin{equation}\label{eq:bl}
 B_{L,x_0}(j^2_{x_0}(v_1),\,j^2_{x_0}(v_2))=
\int_{I_+}\big((D_L v_1)v_2+v_1(D_L v_2)\big)\end{equation}
where $I\subset \CC$ is an open interval around $x_0$, 
$I_+=\{x\in I\colon x\ge x_0\}$ is its positive side with respect to the orientation of $\CC$
and $v_1,v_2$ are representatives of the given 2-jets with compact support in $I$. From now on, we will  write 
$B_{L,x_0}(v_1,v_2)$ in place of $B_{L,x_0}(j^2_{x_0}(v_1),\,j^2_{x_0}(v_2))$. If $x_0,x_1\in \CC$ are the end points of a 
positively oriented arc in $\CC$, we obtain 
\begin{equation}\label{eq:intbyparts}
 \int_{x_0}^{x_1}  \big((D_L v_1)v_2+v_1(D_L v_2)\big)
 =B_{L,x_1}(v_1,v_2)-B_{L,x_0}(v_1,v_2).\end{equation}

\subsection{Virasoro Lie algebra}\label{subsec:vir}
The construction of $\vir(\CC)$ (following \cite{fre:ver}) uses the following well-known principle. 
Let $K$ be a Lie group, with an affine $K$-action on an affine space $E$ such that the underlying linear action is the coadjoint action of $K$ on  $\k^*$. Then one obtains a central extension 
\[ 0\to \R\to \wh{\k}\to \k\to 0\]
 of the Lie algebra. As a vector space, 
$\wh{\k}$ is the space of affine-linear maps $\wh{\xi}\colon E\to \R$; the associated linear map $\xi\colon \k^*\to \R$ is an element of $\k$. The Lie bracket is given by 
\[ [\wh{\xi}_1,\wh{\xi}_2](\mu)=-\l \xi_1\cdot\mu,\xi_2\r,
\]
where $\xi\cdot \mu=\xi_E|_\mu\in T_\mu E=\k^*$ are the infinitesimal 
generators corresponding to $\xi\in\k$. 
A choice of base point $\mu_0\in E$ defines a splitting $\wh{\k}=\k\oplus \R$, and the corresponding Lie algebra cocycle is 
$(\xi_1,\xi_2)\mapsto -\l \xi_1\cdot\mu_0,\xi_2\r$. 

Suppose $\CC$ is compact, connected, and oriented. Using the integration pairing between $|\Omega|^2_\CC$
and $|\Omega|^{-1}_\CC\cong \Vect(\CC)$, the action of $\Diff_+(\CC)$ on quadratic differentials is indeed the coadjoint action on the smooth dual of $\Vect(\CC)$. As remarked above, this is the linear action underlying the affine action of $\Diff_+(\CC)$ on the space of Hill operators. The resulting 
central extension 
\begin{equation}\label{eq:virasoro}
0\to \R\to \vir(\CC)\to \Vect(\CC)\to 0
\end{equation}
is the \emph{Virasoro Lie algebra}. As a vector space, $\vir(\CC)$ consists of affine-linear functionals $\wh{v}\colon \on{Hill}(\CC)\to \R$ for which the underlying linear functional $|\Omega|^{2}_\CC \to \R$ is given by pairing with an element $v\in \Vect(\CC)$.  
Using the description \eqref{eq:infhillaction} of the infinitesimal action, we see that the 
 Lie bracket on $\vir(\CC)$ is  given by 
\begin{equation}\label{eq:bracket}
 [\wh{v}_1,\wh{v}_2](L)
=\int_\CC (D_Lv_1) v_2.\end{equation}
From now on, we refer to the action of $\Diff_+(\CC)$ on the affine space
\begin{equation}\label{eq:coadvir} \vir^*_1(\CC)=\on{Hill}(\CC)\end{equation}
as the \emph{coadjoint Virasoro action}.

\subsection{Coordinate descriptions}\label{subsec:coordinate} 
Choose an identification $\CC\cong S^1=\R/\Z$, and denote by $x$ the local coordinate on $S^1$. The  $r$-density bundles $|\Lambda|^r_\CC$ are canonically trivialized by $|\partial x|^r$ . For $u=f\,|\partial x|^r$, with 
coefficient function $f$, we write 
\begin{equation}\label{eq:primenotation}
u'=\f{\p f}{\p x}  |\partial x|^{r+1},\end{equation}
an $(r+1)$-density. 
The Hill operators on $\CC$ are of the form 
\begin{equation}\label{eq:hillpotential} L u=u''+\T u,\end{equation}
with a \emph{Hill potential} $\T\in |\Omega|^{2}_{S^1}$. For $\FF\in \Diff_+(S^1)$, the Hill potential of 
$\FF^{-1}\cdot L=\FF^*\circ L\circ \FF_*$  is given by
%
%
\begin{equation}\label{eq:coadjointaction}
\FF^*\T+\hh \S(\FF)\end{equation}
where 
$\S\colon \Diff(S^1)\to |\Omega|^{2}_{S^1} $ is the Schwarzian derivative
\begin{equation}\label{eq:schwarzian}\S(\FF)=\f{\FF'''}{\FF'}-\f{3}{2}\big(\f{\FF''}{\FF'}\big)^2.\end{equation}
Here we used \eqref{eq:primenotation} for $\FF$ regarded as an $\R/\Z$-valued function; for example $\FF'''$ is a 3-density. 
The infinitesimal generators of the action are 
\begin{equation}\label{eq:dlv} -D_L v=\T'v+2\T v'+\hh v'''.\end{equation}
One notes that $(D_L v_1)v_2+v_1(D_Lv_2)$ is the total derivative of 
$-\big(2\T v_1v_2+\hh(v_1'' v_2-v_1'v_2'+v_1v_2'')\big)$,
which verifies that $D_L^*=-D_L$, and gives the formula  
\begin{equation}\label{eq:bdryterm}
B_{L,x_0}(v_1,v_2)=\big(2\T v_1v_2+\hh(v_1'' v_2-v_1'v_2'+v_1v_2'')\big)\Big|_{x_0}.
\end{equation}
for the bilinear form \eqref{eq:bl}. 

Taking the Hill operator $L_0$ with the zero Hill potential (i.e., $L_0u=u''$) as the base point for the affine space 
$\on{Hill}(S^1)$, we see that $\vir(S^1)$ is the central extension of $\Vect(S^1)$ defined by the Gelfand-Fuchs cocycle 
\begin{equation}\label{eq:GF} c_{\on{GF}}(v_1,v_2)=\hh \int_{S^1} v_1'''\  v_2,\ \ \ v_1,v_2\in \Vect(S^1).\end{equation}
%
For $\FF\in \Diff_+(S^1)$ and $(\T,\mathsf{c})\in \vir^*(\CC)\cong |\Omega|^{2}_{S^1} \times \R$
we have 
\[ F^{-1}\cdot (\ca{T},\mathsf{c})= \big(\FF^*\T+\f{\mathsf{c}}{2} \S(\FF),\mathsf{c}\big).\] 
%

%

\section{The space of developing maps}\label{sec:dev}
For the rest of this paper, we take $\CC$ to be compact, connected, and oriented (i.e., diffeomorphic to an oriented circle). We let $\wt{\CC}$ be a simply connected covering space, so that $\CC=\wt{\CC}/\Z$.  The diffeomorphism of $\wt{\CC}$ corresponding to translation by $1\in \Z$ will be denoted $\kappa$. 

\subsection{Developing maps}
Let $L\in \on{Hill}(\CC)$ be a Hill operator, and $u_1,u_2\in |\Omega|^{-\f{1}{2}}_{\wt\CC}$ a fundamental system of solutions (defined on the universal cover), with normalized Wronskian 
\[W(u_1,u_2)=-1.\] 
The normalization determines the fundamental system uniquely up to the action of $\SL(2,\R)$. Since $u_1,u_2$ have no common zeroes, 
their ratio
\[ \gamma=(u_1:u_2)\colon \wt{\CC}\to \RP(1)\]
is a well-defined local diffeomorphism. Under the action of $\kappa$, this map transforms according to
$\gamma(\kappa\cdot x)=h\cdot \gamma(x)$ (with the standard action of $\PSL(2,\R)$ on $\RP(1)$),  
for some $h\in \PSL(2,\R)$. 
\begin{definition}\cite{thu:3dim}
	A  \emph{developing map} is an 
	orientation preserving local diffeomorphism $\gamma\colon \wt{\CC}\to \RP(1)$ 
	which is \emph{quasi-periodic}, in the sense that 
	\[ \gamma(\kappa\cdot x) =h\cdot \gamma(x)\]
	for some  \emph{monodromy} $h\in \PSL(2,\R)$. We denote by 
	\[ \DC\subset  C^\infty(\wt{\CC},\RP(1))\]
	the space of developing maps.
\end{definition}
Thus, every normalized fundamental system for a Hill operator defines a developing map; a different choice of 
fundamental system changes $\gamma$ by the natural $\PSL(2,\R)$-action, 
\begin{equation}\label{eq:gaction}
 (g\cdot \gamma)(x)=g\cdot \gamma(x).\end{equation}
Conversely, every developing map $\gamma$ arises in this way from a unique Hill operator $L$. 

\begin{definition}\label{def:normalized}
	A \emph{normalized lift} of $\gamma\in \DC$ is a pair of $-\hh$-densities $u_1,u_2\in |\Omega|^{-\f{1}{2}}_{\wt\CC}$ with $W(u_1,u_2)=-1$, with  $\gamma=(u_1:u_2)$. 
\end{definition}
Every $\gamma$ admits a normalized lift, unique up to an overall sign. 
With $\partial$ as in \eqref{eq:partial}, there is a unique Hill operator 
having the normalized lift as its fundamental system of solutions. Explicitly, 
\begin{equation}\label{eq:Hilloperator} 
L(u)=-\det\left(\begin{array}{ccc}
\partial^*\partial u & \partial^*\partial u_1 & \partial^*\partial u_2\\
\partial u & \partial u_1 & \partial u_2\\
u & u_1 & u_2
\end{array}	
\right)\end{equation}
(The right hand side defines an operator over $\wt\CC$; one checks that it descends to $\CC$.) This defines a surjective map 
\[ p\colon \DC\to \on{Hill}(\CC)=\vir^*_1(\CC),\] 
which is the quotient map for the principal $\PSL(2,\R)$-action \eqref{eq:gaction}.

Writing 
\[ \gamma=(\sin\phi:\cos\phi)\]
with $\phi\in C^\infty(\wt{\CC},\R),\ \phi'>0$ (defined uniquely up to multiples of $\pi$), a normalized lift is given by 
\begin{equation}\label{eq:normalized} u_1=(\partial \phi)^{-\f{1}{2}}\sin\phi,\ u_2=(\partial \phi)^{-\f{1}{2}}\cos\phi\end{equation}
(with $\partial\phi$ the positive 1-density given by exterior differential).
 For $\CC=S^1=\R/\Z$, one finds that the associated Hill potential  is given by 
\begin{equation}\label{eq:T} \T=(\phi')^2 +\hh \S(\phi),\end{equation}
with the Schwarzian derivative $\S(\phi)=\phi'''/\phi' -\f{3}{2}(\phi''/\phi')^2$.

\begin{remark}
Given a developing map, the projective structure on $\RP(1)$ pulls back to a $\Z$-invariant projective structure on $\wt\CC$, which hence descends to $\CC$. This gives the identification $\on{Hill}(\CC)\cong \on{Proj}(\CC)$ of the space of Hill operators with the space of projective structures. 
\end{remark}

\subsection{The quotient map $q$}
The $\PSL(2,\R)$-action on $\RP(1)$ lifts to a $\wt{\SL}(2,\R)$-action on the universal cover, $\wt{\RP}(1)=\R$. 
Hence, writing $\gamma\in \DC$ in the form  $\gamma=(\sin\phi:\cos\phi)$ as above, we have 
	 \[ \phi(\kappa\cdot x)=\wt{h}\cdot \phi(x)\] 
with the \emph{lifted  monodromy} $\wt{h}\in \wt{\SL}(2,\R)$. 
The map $\phi$ is only determined up to multiples of $\pi$, but the lifted monodromy does not depend on the choice. 
This defines a map 
\[ q\colon \DC\to \wt{\SL}(2,\R)\]
taking $\gamma$ to $\wt{h}$. Let $\wt{\SL}(2,\R)_+$ be the range of this map, giving the diagram

\begin{equation}\label{eq:corr0b}
\begin{tikzcd}
[column sep={5.5em,between origins},
row sep={4em,between origins},]
& \DC \arrow[dl, "p"]  \arrow [dr, "q"'] & \\
\vir^*_1(\CC) & & \wt{\SL}(2,\R)_+
\end{tikzcd}
\end{equation}

\begin{proposition}\label{prop:image}
	The image of the map $q$ is the subset 
	\[ \wt{\SL}(2,\R)_+=\{\wt{h}\in\wt{\SL}(2,\R)|\ \exists \phi_0\in \R\colon \wt{h}\cdot \phi_0>\phi_0\}.\]
\end{proposition}
\begin{proof}
	By definition, if $\gamma\in \DC$ has lifted monodromy 
	$\wt{h}=q(\gamma)$, then $\wt{h}\cdot\phi(x)=\phi(\kappa(x))>\phi(x)$ since $\phi$ is increasing. This 
	proves the inclusion $\subseteq$. For the opposite inclusion, suppose that $\wt{h}\in \wt{\SL}(2,\R)$ and $\phi_0\in \R$ with
	$\wt{h}\cdot\phi_0>\phi_0$ are given. Pick a base point $x_0\in\wt{\CC}$, and choose an increasing function 
	$\phi\in C^\infty(\wt{\CC},\R)$ with $\phi(x_0)=\phi_0$ and with the property $\phi(\kappa(x))=\wt{h}\cdot \phi(x)$ for all $x$. 
	Then $\gamma=(\sin(\phi):\cos(\phi))$ is a developing map with lifted monodromy $\wt{h}$.
\end{proof}

For a more concrete description, recall the classification of conjugacy classes in $\wt{\SL}(2,\R)$.
Let 
\begin{equation}\label{eq:j} J_1=\left(\begin{array}{cc}0&1\\ -1& 0
\end{array}\right),\ \ J_2=\left(\begin{array}{cc}1&0\\ 0& -1
\end{array}\right),\ \ J_3=\left(\begin{array}{cc}0&1\\ 0& 0
\end{array}\right),\end{equation}
the basis of $\mf{sl}(2,\R)$ adapted to the standard Iwasawa decomposition $\mf{sl}(2,\R)=\k\oplus \mf{a}\oplus \mf{n}$. 
Every element of $\wt{\SL}(2,\R)$ is conjugate to exactly one of the following types:
\begin{enumerate}
	\item[(i)] {\bf elliptic/central}: $r_\alpha$ for $\alpha\in \R$ 
	is the homotopy class of the path $t\mapsto \exp(t\alpha J_1)$. 
	In particular,  $r_{\pi n}, \, n\in \Z$ are the central elements of $\wt{\SL}(2,\R)$. \smallskip
	\item[(ii)] {\bf hyperbolic}: 
	$h_{\beta,n}=r_{\pi n} h_{\beta,0}$ for $n\in \Z,\ \beta>0$, 
	where 
	$h_{\beta,0}\in \wt{\SL}(2,\R)$ is the homotopy class of the path $t\mapsto \exp(t \beta J_2)$.\smallskip
	\item[(iii)] {\bf parabolic}: $p_n^\pm=r_{\pi n} p_0^\pm$ where 
	$p_0^\pm$ are represented by paths $t\mapsto \exp(\pm t J_3)$.
\end{enumerate}
Among these elements, exactly 
\begin{equation}\label{eq:representatives}
 \{r_\alpha|\ \alpha>0\}\cup \{h_{\beta,n}|\ \beta>0,\ n\ge 0\}\cup \{p^\pm_n|\ n>0\}\cup \{p^+_0\}\end{equation}
are in $\wt{\SL}(2,\R)_+$. The (non-Hausdorff) space $\wt{\SL}(2,\R)/\PSL(2,\R)$ of conjugacy classes is illustrated by the following well-known picture.
\bigskip
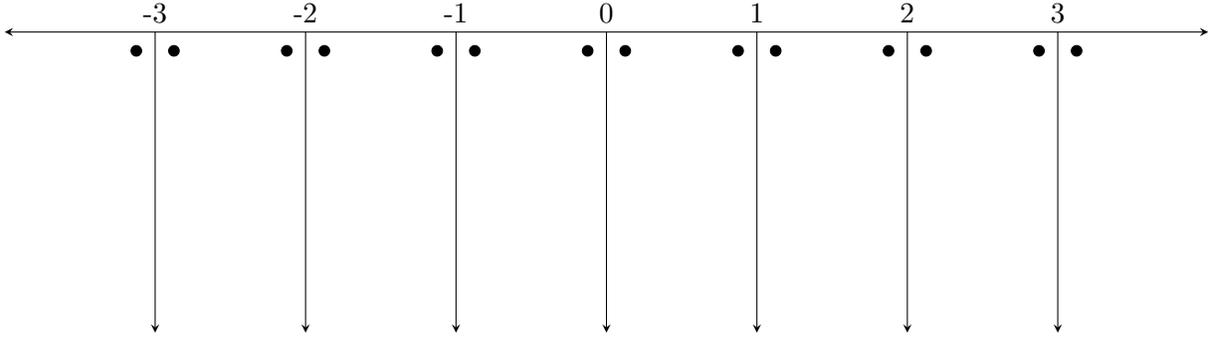
\begin{center}
		\begin{tikzpicture}
		\draw [stealth-stealth](-8,0) -- (8,0);
		\draw [-stealth](0,0) --(0,-4);
		\draw [-stealth](2,0) --(2,-4);
		\draw [-stealth](4,0) --(4,-4);
		\draw [-stealth](6,0) --(6,-4);
		\draw [-stealth](-2,0) --(-2,-4);
		\draw [-stealth](-4,0) --(-4,-4);
		\draw [-stealth](-6,0) --(-6,-4);
		\draw (0,0.25) node{0};
		\draw (2,0.25) node{1};
		\draw (4,0.25) node{2};
		\draw (6,0.25) node{3};
		\draw (-2,0.25) node{-1};
		\draw (-4,0.25) node{-2};
		\draw (-6,0.25) node{-3};
		\filldraw[black] (-0.25,-0.25) circle (2pt);
		\filldraw[black] (0.25,-0.25) circle (2pt);
			\filldraw[black] (1.75,-0.25) circle (2pt);
		\filldraw[black] (3.75,-0.25) circle (2pt);
		\filldraw[black] (5.75,-0.25) circle (2pt);
		\filldraw[black] (2.25,-0.25) circle (2pt);
		\filldraw[black] (4.25,-0.25) circle (2pt);
		\filldraw[black] (6.25,-0.25) circle (2pt);
		\filldraw[black] (-1.75,-0.25) circle (2pt);
		\filldraw[black] (-3.75,-0.25) circle (2pt);
		\filldraw[black] (-5.75,-0.25) circle (2pt);
		\filldraw[black] (-2.25,-0.25) circle (2pt);
		\filldraw[black] (-4.25,-0.25) circle (2pt);
		\filldraw[black] (-6.25,-0.25) circle (2pt);
		\end{tikzpicture}
\end{center}
\bigskip

The horizontal line in this picture depicts the elliptic/central classes, with the central elements $r_{\pi n}$ as the nodes $n$. The vertical lines represent the hyperbolic classes, and the fat dots represent the parabolic classes. 
The space $\wt{\SL}(2,\R)_+/\PSL(2,\R)$ is the subset\medskip

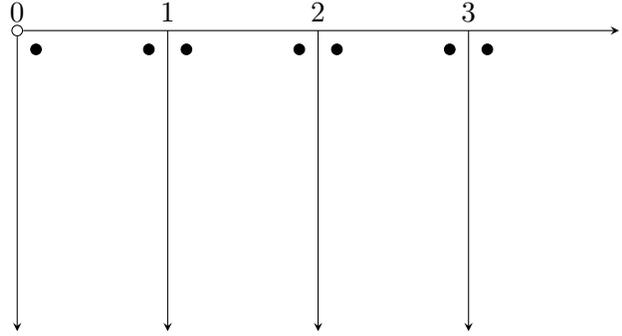
\hfill	\begin{tikzpicture}
\draw [-stealth](0.08,0) -- (8,0);
\draw [-stealth](0,-0.08) --(0,-4);
\draw [-stealth](2,0) --(2,-4);
\draw [-stealth](4,0) --(4,-4);
\draw [-stealth](6,0) --(6,-4);
\draw (0,0.25) node{0};
\draw (2,0.25) node{1};
\draw (4,0.25) node{2};
\draw (6,0.25) node{3};
\draw(0,0)  circle (2pt);
\filldraw[black] (0.25,-0.25) circle (2pt);
\filldraw[black] (1.75,-0.25) circle (2pt);
\filldraw[black] (3.75,-0.25) circle (2pt);
\filldraw[black] (5.75,-0.25) circle (2pt);
\filldraw[black] (2.25,-0.25) circle (2pt);
\filldraw[black] (4.25,-0.25) circle (2pt);
\filldraw[black] (6.25,-0.25) circle (2pt);
\end{tikzpicture}

\bigskip

	Let $\CC=S^1=\R/\Z$, so $\wt\CC=\R,\ \kappa(x)=x+1$. 
The article of Balog-Feher-Palla \cite{bal:coa} (see also \cite{dai:coa}) contains a detailed discussion of the 
developing maps and Hill potentials for the representatives \eqref{eq:representatives}. 
 The simplest examples are given by 	`exponential paths': 
 \[ x\mapsto \exp(xA)\cdot (\sin\phi_0:\cos\phi_0)\] 
 for $A\in \mf{sl}(2,\R),\ \phi_0\in\R$:
 		\begin{align*}
 	\gamma(x)&=\exp(x\alpha J_1)\cdot (0:1)=(\sin(\alpha x):\cos(\alpha x)) & &q(\gamma)=r_\alpha, 
 	&& \T=\alpha^2 |\partial x|^2
 	\\
 	\gamma(x)&=\exp(x\beta J_2)\cdot (1:1)=( e^{\beta x}:\ e^{-\beta  x})& &q(\gamma)=h_{\beta ,0},\ && \T=-\beta^2 |\partial x|^2\\
 	\gamma(x)&=\exp(xJ_3)\cdot(0:1)= (x:1)& & q(\gamma)=p_0^+&& \T=0
 	\end{align*}
 (Here $\alpha,\beta>0$.) 	
 	For fixed $n\in \N$, if $A\in\mf{sl}(2,\R)$ is sufficiently close to $0$,  the path 
 	\begin{equation} \gamma(x)=\exp(x A)\exp(x\pi n J_1)\cdot (0:1)\end{equation}
 is a developing map (The condition $\gamma'>0$ holds since it is a small perturbation of the path for $A=0$).  This may be used to realize all $\wt{h}$ 
 in some open neighborhood of the central element $r_{\pi n}$.  
 In particular, the conjugacy classes of $p_n^\pm,\ n\in \N$  may be realized in this way, as well as $h_{\beta,n}$ for $\beta $ small. 	As developing maps realizing $h_{\beta,n}$ with $n\in \N$ for arbitrary $\beta >0$  one may take \cite[Section 3.5.2]{bal:coa}
 	\begin{align}\label{ex:feher}\gamma(x)&=\exp(x \beta   J_2)\exp(\f{\beta }{\pi n} J_3)\exp(x \pi n J_1)\cdot (0:1)
 	& &q(\gamma)=h_{\beta ,0},\ \ \beta >0,\ n\in \N.
 	\end{align}
 

%

\subsection{Action of diffeomorphisms}
For the following discussion, we shall use the \emph{Poincar\'{e} translation number} 
\cite{ghy:gro}. The translation number is a quasi-homomorphism $\tau\colon \wt{\SL}(2,\R) \to \R$, 
defined in terms of the action of $\wt{\SL}(2,\R)$ on $\R=\wt{\RP}(1)$ by 
\begin{equation}\label{eq:translationnumber} \tau(\wt{g})=\lim_{N\to \pm\infty} \f{\wt{g}^N.\phi_0}{N\pi},  
\end{equation}
for arbitrary $\phi_0\in \R$. This is well-defined and independent of the choice of $\phi_0$; it is furthermore conjugation invariant and continuous, and it satisfies the quasi-homomorphism property $|\tau(\wt{g}_1\wt{g}_2)-\tau(\wt{g}_1)-\tau(\wt{g}_2)|\le 1$. 
One may think of the $\tau(\wt{g})$ 
as the `asymptotic slope' of the action of $\ti{g}$ on $\R$. We have 
\[ \tau(r_\alpha)=\f{1}{\pi} \alpha,\ \ \ \tau(h_{\beta,n})=n,\  \ \ \tau(p^\pm_n)=n.\] 
In terms of the picture of conjugacy classes, $\tau$ is simply the horizontal coordinate.

Recall that the action of $h\in \PSL(2,\R)$ on $\RP(1)$ has two fixed points if $h$ is hyperbolic, and a single fixed point if $h$ is parabolic. 

\begin{lemma}\label{lem:image1}
	Let $\gamma\in \DC$ be a developing map with lifted monodromy $\wt{h}=q(\gamma)$. 
	\begin{enumerate}
		\item If $\tau(\wt{h})>0$, then $\gamma\colon \R\to \RP(1)$ is surjective (and hence is a covering). 
		\item If $\tau(\wt{h})=0$, and $h$ is hyperbolic, then the range of $\gamma$ is an open arc whose end points are the 
		 two fixed points of $h$.
		\item  If $\tau(\wt{h})=0$, and $h$ is parabolic, then the range of $\gamma$ is the complement of the unique fixed point of $h$. 
\end{enumerate}
\end{lemma}
\begin{proof}
	Write $\gamma=(\sin\phi:\cos\phi)\in \DC$.  We shall use the property 
	\begin{equation}\label{eq:see} \phi(\kappa^N(x))=\wt{h}^N\cdot \phi(x)\end{equation}
	for $x\in \R$, $N\in \Z$. 
	If $\tau(\wt{h})>0$, this property shows that $\phi$ is surjective, and hence 
	is a global diffeomorphism. Suppose  $\tau(\wt{h})=0$. In particular, $h$ must be hyperbolic or parabolic, and so its action 
	on $\RP(1)$ has (one or two) fixed points. Their pre-images are fixed points for the action of $\wt{h}$ (since 
	 $\tau(\wt{h})=0$). Since $\phi$ is increasing, equation \eqref{eq:see} shows that the range of $\phi$ cannot contain any such fixed points, but that for any given $x\in\wt\CC$ the limits $\lim_{N\to \pm\infty}\phi(\kappa^N(x))$ are fixed points. Hence, the range of $\phi$ must be the interval between two successive fixed points. 
\end{proof}


Let $\Diff_\Z(\wt\CC)$ be the $\Z$-equivariant diffeomorphisms of $\wt\CC$. 
Any such diffeomorphism descends to an orientation preserving diffeomorphism of $\CC$, resulting in an identification 
\[ \Diff_\Z(\wt\CC)\cong \wt\Diff_+(\CC)\]
with the 
universal covering group of $\Diff_+(\CC)$. 

\begin{proposition}
	The map $q\colon  \DC\to \wt{\SL}(2,\R)_+$ is the quotient map for the action 
	\[ (\FF\cdot \gamma)(x)=\gamma(\FF^{-1}(x))\]
	of $\Diff_\Z(\wt\CC) $ on the space of developing maps. 
\end{proposition}
\begin{proof}
	Let $\gamma_i=(\sin\phi_i:\cos\phi_i),\  i=1,2$ be two developing maps 
	with $q(\gamma_1)=q(\gamma_2)=\wt{h}$. Since the $\phi_i$ are increasing functions, they are  diffeomorphisms onto their range. By
	Lemma \ref{lem:image1}, the maps $\phi_1,\phi_2$ have the same range, after possibly changing $\phi_2$ by a multiple of $\pi$. 
	But then $\FF=\phi_2^{-1}\circ \phi_1\in \Diff_\Z(\wt\CC)$ satisfies	$\gamma_2=\gamma_1\circ \FF^{-1}$. 
\end{proof}
To summarize, the maps $p$ and $q$ in the diagram \eqref{eq:corr0b} are the 
quotient maps for commuting actions of $\PSL(2,\R)$ and $\Diff_\Z(\wt\CC)$, respectively. The projection $p\colon  \DC\to \vir^*_1(\CC)$ intertwines the $\on{Diff}_\Z(\wt\CC)$-action  with the action 
of $\Diff_+(\CC)$ on $\on{Hill}(\CC)$. 
We therefore obtain an identification of orbit spaces 
\begin{equation} \vir^*_1(\CC)/\Diff_+(\CC)\cong \wt{\SL}(2,\R)_+/\PSL(2,\R)\end{equation}
since both are identified as the quotients of $\DC$ under the action of $\Diff_\Z(\wt\CC)\times \PSL(2,\R)$. This recovers the classification of 
coadjoint Virasoro orbits in the form due to Goldman and Segal 
\cite{gol:the,seg:uni}. 

Nonetheless,  $\DC$ does \emph{not} give a Morita equivalence of action groupoids
for the two actions, since it does not identify the stabilizer groups. This may be traced back to the fact that 
the $\Diff_\Z(\wt\CC)$-action on the space of developing maps is not free: 
\begin{example}\label{ex:nontrivial}
	Let $\CC=S^1=\R/\Z$, and let $\gamma=(\sin(\alpha x):\cos(\alpha x)),\ \alpha>0$ be the developing map  realizing the elliptic/central element $r_\alpha$. The 
	stabilizer of $\gamma$ is 
	\[\Diff_\Z(\R)_\gamma= \f{\pi}{\alpha}\Z\subset  \Diff_\Z(\R),\]
	acting by translations. 
\end{example}

\subsection{Stabilizers} 
To understand the general situation, we shall now determine the stabilizers for the 
$\Diff_\Z(\wt\CC)\times \PSL(2,\R)$-action on $\DC$, and relate them to the stabilizers of the action of $\Diff_+(\CC)$ in $\vir^*_1(\CC)$ 
 and the conjugation action of $\PSL(2,\R)$ on $\wt{\SL}(2,\R)_+$.  Consider $\gamma\in \DC$, with monodromy $h\in \PSL(2,\R)$, and let $L=p(\gamma),\ \wt{h}=q(\gamma)$. The stabilizer group $(\Diff_\Z(\wt\CC)\times \PSL(2,\R))_\gamma$ consists of all 
 $(\FF,g)$ such that
 \[ \gamma(\FF(x))=g\cdot \gamma(x)\]
 for all $x\in \wt\CC$. 	 Since the $\PSL(2,\R)$-action on $\D(\CC)$ is free, the map 
 $(\FF,g)\mapsto \FF$ induces an isomorphism 
 \begin{equation}\label{eq:general1}  (\Diff_\Z(\wt\CC)\times \PSL(2,\R))_\gamma\cong \Diff_\Z(\wt\CC)_L.\end{equation}
The pre-image of $\kappa\in \Diff_\Z(\wt\CC)_L$ under this map is $(\kappa,h)$. Hence, 
\begin{equation}\label{eq:general4}
 \Diff_+(\CC)_L=(\Diff_\Z(\wt\CC)\times \PSL(2,\R))_\gamma/\Z\end{equation}
where $\Z$ is the subgroup generated by $(\kappa,h)$.
On the other hand, the projection $(\FF,g)\mapsto g$ induces a surjection
\begin{equation}\label{eq:general3}  (\Diff_\Z(\wt\CC)\times \PSL(2,\R))_\gamma\longrightarrow \PSL(2,\R)_h\end{equation}
with kernel  $\Diff_\Z(\wt\CC)_\gamma$. (Note that the stabilizer of $\wt{h}$ equals that of its image $h$.) 
	That is, 
\begin{equation}\label{eq:general2}\PSL(2,\R)_h=\Diff_\Z(\wt\CC)_{L}/\Diff_\Z(\wt\CC)_\gamma.\end{equation}
With these preparations, we obtain the following description of the stabilizer groups:

\begin{proposition}\label{prop:stabilizers}
 Let $\gamma\in\DC$ be as above, written as $\gamma=(\sin\phi:\ \cos\phi)$.
\begin{itemize}
	\item If  $\tau(\wt{h})>0$, there is an isomorphism of stabilizer groups 
\[ 	\wt{\SL}(2,\R)_{\wt{h}}\stackrel{\cong}{\lra} 
(\Diff_\Z(\wt\CC)\times \PSL(2,\R))_\gamma
,\ \ \ \wt{g}\mapsto (\FF_{\wt{g}},g)\]
with the unique  $\FF_{\wt{g}}$ such that
$\phi\circ \FF_{\wt{g}}=\wt{g}\cdot \phi$. It induces an isomorphism 
\[ \Diff_+(\CC)_{L}\cong \wt{\SL}(2,\R)_{\wt{h}}/\{\wt{h}^n,\ n\in\Z\}.\] 
The stabilizer group $\Diff_\Z(\wt\CC)_\gamma$  is a cyclic group generated by the unique diffeomorphism $\FF$ such that 
$\phi(\FF(x))=\phi(x)+\pi$.

\item If   $\tau(\wt{h})=0$, there is an isomorphism  
\[ \PSL(2,\R)_h\stackrel{\cong}{\lra} (\Diff_\Z(\wt\CC)\times \PSL(2,\R))_\gamma
,\ \ \ g\mapsto (\FF_g,g)\]
	where $\FF_g$ is defined as $\phi\circ \FF_g=\wt{g}\cdot \phi$, with 
 $\wt{g}$ the unique lift of translation number $\tau(\wt{g})=0$. 
	It induces an isomorphism 
\[\Diff_+(\CC)_{L}\cong \PSL(2,\R)_{h}/\{ h^n,\ n\in\Z\}.\] 
The stabilizer group $\Diff_\Z(\wt\CC)_\gamma$  is trivial. 
\end{itemize}
\end{proposition}
\begin{proof}
\begin{itemize}
	\item Case 1:  $\tau(\wt{h})>0$, so $\phi$ is a diffeomorphism onto $\R$.  
	If $(\FF,g)$ stabilizes $\gamma$, there is a 
	unique lift $\wt{g}$ of $g$ such that 
	\begin{equation}\label{eq:fdef}  \phi(\FF(x))=\wt{g}\cdot \phi(x).\end{equation}
	Conversely, given  $\wt{g}\in \wt{\SL}(2,\R)_{\wt{h}}$,
	use \eqref{eq:fdef} to define 
	$\FF$. This is well-defined since $\phi$ is a diffeomorphism, and 
	$\FF$ is periodic by the calculation
	\[ \phi(\FF(\kappa(x)))=\wt{g}\cdot \phi(\kappa(x))=\wt{g}\wt{h}\cdot \phi(x)=\wt{h}\wt{g}\cdot \phi(x)=\wt{h}\cdot \phi(\FF(x))
	=\phi(\kappa(\FF(x)).\]
	We hence obtain an isomorphism $\Diff_\Z(\wt\CC)_L\cong (\Diff_\Z(\wt\CC)\times \PSL(2,\R))_\gamma\cong \wt{\SL}(2,\R)_{\wt{h}}$. This isomorphism takes $\kappa$ to 
	$\wt{h}$. Using \eqref{eq:general3}, this gives the description of 
	$\Diff_+(\CC)_{L}$.
	
	\item Case 2: $\tau(\wt{h})=0$, so $\phi$ is a diffeomorphism onto an open interval between two fixed points of 
	the action of $\wt{h}$ on $\R$, of length $\pi/2$ if $h$ is hyperbolic and $\pi$ if $h$ is parabolic. 
	If $(\FF,g)$ stabilizes $\gamma$, there is, as 
	before, a unique lift $\wt{g}$ such that \eqref{eq:fdef} holds. 
	But this equation means, in particular, that $\wt{g}$ preserves the range of $\phi$, and so it has translation number zero.  
 Conversely, given $g\in \PSL(2,\R)_h$, let $\wt{g}$ be the unique lift of translation number $0$. The equation \eqref{eq:fdef} uniquely determines $\FF$. 	The same calculation as in Case 1 shows that $\FF$ is periodic. 
	Then $(\FF,g)$ stabilizes $\gamma$. 
	
The resulting isomorphism  $\Diff_\Z(\wt\CC)_L\cong (\Diff_\Z(\wt\CC)\times \PSL(2,\R))_\gamma\cong \PSL(2,\R)_h$ takes 
	$\kappa$ to $h$. Using \eqref{eq:general3}, this gives the description of 
	$\Diff_+(\CC)_{L}$.
\end{itemize}

The description of $\Diff_\Z(\wt\CC)_\gamma$ follows from the above, since the stabilizer consists of all $\FF$ such that $ (\FF,e)\in  (\Diff_\Z(\wt\CC)\times \PSL(2,\R))_\gamma$.
\end{proof}

\begin{example}
If $\wt{h}=r_{\pi n}$ is a central element, then $\wt{\SL}(2,\R)_{\wt{h}}=\wt{\SL}(2,\R)$. Taking the quotient by the subgroup generated by $\wt{h}$,  we see that the stabilizer $\Diff_+(\CC)_L$ is the $n$-fold cover of $\PSL(2,\R)$. If $\wt{h}=r_{\alpha}$ with $\alpha>0,\ \alpha\neq n\pi$, then $\wt{\SL}(2,\R)_{\wt{h}}=\wt{\SO}(2)=\R$. Taking the quotient by the subgroup generated by $\wt{h}$,  we see that the stabilizer $\Diff_+(\CC)_L$ is $\R/\alpha\Z\cong \SO(2)$. In a similar way, one obtains explicit descriptions of the stabilizers for hyperbolic and elliptic elements, matching the calculations of  
\cite{bal:coa}.
\end{example}

We have seen that $\Diff_\Z(\wt\CC)_\gamma$ is a cyclic subgroup if $\tau(\wt{h})>0$, and trivial if $\tau(\wt{h})=0$. 
For a better understanding of how these stabilizer groups fit together, take  $\CC=S^1=\R/\Z$, and let
\[ \tau(\FF)=\lim_{N\to \infty}\Big(\f{1}{N}\FF^N(x_0)\Big)\]
be the Poincar\'{e} translation number \cite{ghy:gro} of $\FF\in \on{Diff}_\Z(\R)$. This is independent of $x_0$, and we shall choose $x_0=0$. 

Consider $\gamma=(\sin\phi: \cos\phi)\in \DC$, with lifted monodromy $\wt{h}=q(\gamma)$ satisfying 
$\tau(\wt{h})>0$. Let 
$\FF\in \on{Diff}_\Z(\R)_\gamma$ be the generator of the stabilizer of $\gamma$, so that 
\[ \phi(\FF(x))=\phi(x)+\pi.\] 
We will show
\begin{equation}\label{eq:product} \tau(\wt{h})\tau(\FF)=1.\end{equation}
To see this, observe first that $\tau(\wt{h})$ may be computed as
\[ \tau(\wt{h})=\lim_{N\to \infty}\f{\wt{h}^N\cdot \phi(0)}{\pi N}=\lim_{N\to \infty}\f{\phi(N)}{\pi N}=\lim_{N\to \infty}\f{\phi(x_N)}{\pi x_N}\]
where $x_N$ is any sequence with $\lim_{N\to \infty}x_N=\infty$. (The last equality follows since $\phi$ is increasing.) In particular, we may take $x_N=F^N(0)$, 
which gives 
\[ 1=\lim_{N\to \infty}\f{\phi(0)+\pi N}{\pi N}=
\lim_{N\to \infty}\f{\phi(\FF^N(0))}{\pi N}
=\lim_{N\to \infty}\f{\phi(\FF^N(0))}{\pi\FF^N(0)}\f{\FF^N(0)}{N}=
\tau(\wt{h})\tau(\FF)\]
as claimed. 
This is consistent with Example \ref{ex:nontrivial}. 

\begin{example}\label{ex:babyexample}
	It may be puzzling that the $\Diff_\Z(    \wt\CC)$-action on $\DC$ results in a 
	smooth orbit space $\wt{\SL}(2,\R)_+$, despite having non-constant stabilizers. This is possible since the action is not proper.  For a finite-dimensional example, let $N=\R^2/\sim$ be the quotient under the equivalence relation generated by $(x,y)\sim (x,y+\f{1}{x})$ for $x\neq 0$. 
	By the standard criterion for quotients of manifolds, $N$ is a smooth Hausdorff manifold; one may picture it as a union of circles $\R/\f{1}{x}\Z$, which for $x=0$ degenerates to $\R$.  
	The $\R$-action $t\cdot (x,y)=(x,y-t)$ on $\R^2$  descends to a locally free $\R$-action on $N$, 
	with orbit space $N/\R=\R$. The stabilizer of $[(x,y)]\in N$ under this action is trivial when $x=0$, and equal to $\f{1}{x}\Z$ when 
	$x\neq 0$. This space $N$ will appear again in Appendix \ref{app:a}, as an example of a symplectic groupoid. 
\end{example}

\subsection{Local sections} 
To complete our discussion of the geometry of the space of developing maps, let us show that the 
 quotient map $q\colon \DC\to \wt{\SL}(2,\R)_+$  is a submersion, in the sense that it admits local sections. We may assume $\CC=S^1=\R/\Z$. Local sections can be obtained by the following observations:
\begin{enumerate}
	\item For fixed $\phi_0\in \R$, with image $\gamma_0\in \RP(1)=\R/\pi\Z$, let 
	\[ U=\{\wt{h}\in \wt{\SL}(2,\R)|\ \wt{h}\cdot\phi_0>\phi_0,\ \ \tau(\wt{h})<1\}.\]
	The first condition shows that $U$ is contained in $\wt{\SL}(2,\R)_+$ (see Proposition \ref{prop:image}), while the second condition ensures that $\wt{h}$ is the homotopy  class of a unique exponential path $t\mapsto \exp(tA)$, for $A\in \g$. 
	Hence, the map taking $\wt{h}\in U$ to the path $\gamma(x)=\exp(xA)\cdot \gamma_0$,  defines a section 
	of $q$ over $U$. The $U$'s of this form (for choices of $\phi_0$) cover the subset of $\wt{\SL}(2,\R)_+$ given by $0\le \tau(\wt{h})<1$. 
	\item We obtain similar sections $\gamma(x)=\exp(xA)\cdot (0:1)$ 
	over  each component  of the set of elliptic elements, using the fact that every elliptic element  
	lies in a unique 1-parameter subgroup. 
	\item Paths of the form $\gamma(x)=\exp(x A)\exp(\pi n x)\cdot (0:1)$, with $A\in\g$ sufficiently close to $0$, 
	define a section over an open neighborhood of the central element $r_{\pi n}\in \wt{\SL}(2,\R)$.  
	\item 
	Finally, a section on an open neighborhood of $\{h_{\beta,n},\, \beta>0\}$ for fixed  $n\in \N$ is given by a modification of 
	\eqref{ex:feher}: One takes paths 	
	\[\gamma(x)=\exp(x A)\exp(\f{\beta}{\pi n} J_3)\exp(x \pi n J_1)\cdot (0:1)
	\]
	where $A\in\mf{sl}(2,\R)$ are matrices with $\det(A)=-\beta^2<0$, sufficiently close to a diagonal matrix of the form 
	$\on{diag}(\beta,-\beta)$ with $\beta>0$. 
\end{enumerate}

\section{Morita equivalence for the  coadjoint Virasoro action}\label{sec:mor}
We shall now exhibit $\DC$ as a bimodule for a Morita equivalence of quasi-symplectic groupoids. Background on Morita equivalence of groupoids, and the extension to quasi-symplectic groupoids, is provided in Appendix \ref{app:a}.  

\subsection{Groupoids}\label{subsec:actiongroupoidvirasoro}
By applying Example \ref{ex:relevant} to the principal $\PSL(2,\R)$-bundle $p\colon \DC\to \vir^*_1(\CC)$ with the $\PSL(2,\R)$-equivariant submersion 
$q\colon \DC\to \wt{\SL}(2,\R)_+$, we obtain a Morita equivalence 
\begin{equation}\label{eq:corr0c}
\xymatrix{ \G_1\ar[d]<2pt>\ar@<-2pt>[d]& \DC\ar[dl]\ar[dr] & \G_2 \ar[d]<2pt>\ar@<-2pt>[d]\\
	\vir^*_1(\CC) & & \wt{\SL}(2,\R)_+
}\end{equation}
between the action groupoid $\G_2=\PSL(2,\R)\ltimes \wt{\SL}(2,\R)_+\rra \wt{\SL}(2,\R)_+$, and the groupoid
\begin{equation}\label{eq:g1}
 \G_1=(\DC\times_{\wt{\SL}(2,\R)_+}\DC)/\PSL(2,\R)\rra \vir^*_1(\CC).\end{equation}
The groupoid $\G_1$, presented here as a quotient by $\PSL(2,\R)$ of the submersion groupoid, 
 may also be seen as a quotient of the action groupoid 
 \[ \wt{\G}_1=\Diff_\Z(\wt\CC)\ltimes \vir^*_1(\CC)\rra \vir^*_1(\CC).\]
\begin{proposition}
There is an exact sequence of groupoids over $\vir^*_1(\CC)$, 
\begin{equation}\label{eq:surjective}
1\to \ca{K}\to \wt{\G}_1\to \G_1\to 1
\end{equation}
where $\ca{K}$ is the family of discrete groups 
\begin{equation}\label{eq:ksubgroupoid} 
\ca{K}_L=\Diff_\Z(\wt\CC)_\gamma,\ \ \ p(\gamma)=L.\end{equation}
\end{proposition}Recall again  (Proposition \ref{prop:stabilizers})  that \eqref{eq:ksubgroupoid} is trivial when $\tau(q(\gamma))=0$, and is a cyclic subgroup otherwise. 
\begin{proof}
By definition, $\G_1$ consists of pairs of developing maps  having the same lifted monodromy, 
up to simultaneous transformation by $\PSL(2,\R)$. Since $\Diff_\Z(\wt\CC)$ acts transitively on the fibers of $q$, this defines a 
surjective groupoid morphism  $\wt{\G}_1\to \G_1$ taking $(\FF,p(\gamma))$ to the $\PSL(2,\R)$-orbit of $(\FF\cdot\gamma,\gamma)$. The kernel of this groupoid morphism is $\ca{K}$. 
\end{proof}

\begin{remark}\label{rem:hyperbolic}
The picture simplifies when we restrict to the open subset $\on{Hyp}_0\subset \wt{\SL}(2,\R)_+$ consisting of hyperbolic elements of translation number $\tau=0$. Letting  $\ca{R}_0\subset \vir^*_1(\CC)$ be the corresponding invariant subset, we obtain a Morita equivalence between the action groupoids $\Diff_\Z(\wt{\CC})\ltimes \ca{R}_0\rra \ca{R}_0$ and $\PSL(2,\R)\times \on{Hyp}_0\rra \on{Hyp}_0$.
\end{remark}

\subsection{The 2-form $\varpi_\D$}
In this section, we will promote \eqref{eq:corr0c} to a Morita equivalence of quasi-symplectic groupoids. This involves a 2-form $\varpi_\D$ 
on $\DC$, which is `natural' in the sense that its definition does not involve choices, and in particular does not depend on a  parametrization of $\CC$. Throughout, we will work with the metric
	$\l X,Y\r=2\on{tr}(XY)$ on $\g$.

\begin{theorem}\label{th:propertiesofvarpi}
There is a natural 2-form $\varpi_\D\in \Omega^2(\DC)$, with the following properties:
 \begin{enumerate}
 	\item The exterior differential is 
 	\[ \d\varpi_\D=q^*\eta\]
 	where $\eta=\f{1}{6} \on{tr}(\theta^L\wedge [\theta^L,\theta^L])
 	$ is the Cartan 3-form on $\wt{\SL}(2,\R)$.
 	\item $\varpi_\D$ is $\PSL(2,\R)$-invariant, and its  contractions with generating vector fields for $X\in\mf{sl}(2,\R)$ 
 	are 	\[ \iota(X_\D)\varpi_\D=\on{tr}( X\, q^*(\theta^L+\theta^R)).\]
 	\item  $\varpi_\D$ is $\Diff_\Z(\wt\CC)$-invariant, and its  contractions with generating vector fields for 
 	$v\in \Vect(S^1)$   are	
 	\begin{equation} \iota(v_\D)\varpi_\D=-p^* \int_{\CC}\ (\d L)\,v.\end{equation}
 \end{enumerate} 
 \end{theorem}
The proof of this result is contained in the subsequent sections. Proposition \ref{prop:independence} presents a coordinate-free formula for $\varpi_\D$, 
which in particular gives its invariance properties. Theorem \ref{th:pullback} exhibits $\varpi_\D$ as the pullback of a similar 2-form $\varpi_\P$ for loop groups; the formulas for contractions, and the exterior differential, are immediate consequences of the properties of $\varpi_\P$. In Section \ref{subsec:ds} we will give a conceptual explanation of 
this description, as a Drinfeld-Sokolov reduction. 

\begin{corollary}\label{cor:symplectic}
	The 2-form $\omega_{can}$ of the symplectic groupoid $\wt{\G}_1\rra \vir^*_1(\CC)$ (cf.~ Example \ref{ex:centralext})	descends to a symplectic form 
	$\omega_1$ on $\G_1$.
\end{corollary}
\begin{proof}
As recalled in Appendix \ref{app:a}, a closed multiplicative 2-form $\omega$ on a Lie groupoid $\G\rra M$ determines a map 
$\rho\colon A=\on{Lie}(\ca{G})\to T^*M$, given on the level of sections by $\rho(\sigma)=i_M^*\iota(\sigma^L)\omega$ where 
$\sigma^L$ is the left-invariant vector field  corresponding to $\sigma$, and $i_M$ is the inclusion of units. 
If $\G$ is source-connected, $\omega$ is uniquely determined by $\rho$. 

The Lie algebroid of $\wt{\G}_1$ is the action algebroid
$A=\Vect(\CC)\ltimes \vir^*_1(\CC)$, and the map $\varrho$ is given on constant sections by 
\begin{equation}\label{eq:varrho} \varrho(v)=\l dL,\,v\r=\int_\CC (d L)\, v.\end{equation}
(See Example \ref{ex:centralext}.) The groupoid $\G_1$ has the same Lie algebroid, $A$. By the above, there  can be at most one closed multiplicative form $\omega_1\in \Omega^2(\G_1)$ having the same map \eqref{eq:varrho}, and its pullback to  $\wt{\G}_1$ must coincide  with $\omega_{can}$. In particular, $\omega_1$ must be symplectic. 

To construct $\omega_1$ we use the description \eqref{eq:g1} of the groupoid. 
The 2-form $\pr_1^*\varpi_\D-\pr_2^*\varpi_\D$ on 
\begin{equation}\label{eq:subgroupoid}
\DC\times_{\wt{\SL}(2,\R)_+} \DC\rra \DC.
\end{equation}
(where $\pr_1,\pr_2$ are the projections from the fiber product to the two factors)
is multiplicative, and 
\[ \d(\pr_1^*\varpi_\D-\pr_2^*\varpi_\D)=-\pr_1^*q^*\eta+\pr_2^*q^*\eta=0\]
since $q\circ \pr_1=q\circ \pr_2$ by definition of the fiber product. 
This 2-form is invariant under the diagonal $\PSL(2,\R)$-action, and in fact is $\PSL(2,\R)$-basic since 
\[ \iota(X_\D,X_\D)(\pr_1^*\varpi_\D-\pr_2^*\varpi_\D)=\on{tr}\big( X\, (\pr_1^*q^*(\theta^L+\theta^R)-\pr_2^*q^*(\theta^L+\theta^R))\big)=0\]
It hence descends to a closed multiplicative 2-form $\omega_1\in \Omega^2(\G_1)$. The left invariant vector field on \eqref{eq:subgroupoid} corresponding to $v\in \Vect(\CC)$ is given by $(0,v_\D)$, and descends to the corresponding vector field on $\G_1$. By part (c) of Theorem \ref{th:propertiesofvarpi} we have 
\[ \iota\big((0,v_\D)\big)(\pr_1^*\varpi_\D-\pr_2^*\varpi_\D)=\pr_2^* p^* \int_\CC (d L)v.\]
This confirms that the map $\varrho$ for $\omega_1$ is given by \eqref{eq:varrho}.
\end{proof}	

Corollary  \ref{cor:symplectic} shows that $\G_1\rra \vir^*_1(\CC)$ is a symplectic groupoid.  On the other hand, $\PSL(2,\R)\ltimes \wt{\SL}(2,\R)\rra \wt{\SL}(2,\R)$ comes equipped with a 2-form $\omega_2$ making it a quasi-symplectic groupoid integrating the 
Cartan-Dirac structure \cite{al:mom,xu:mom} (see Example \ref{ex:gvalued} in the appendix).  
The groupoid $\G_2$ is simply its restriction to the invariant open subset $\wt{\SL}(2,\R)_+$. Using our construction of $\omega_1$, 
the following is  a reformulation of Theorem \ref{th:propertiesofvarpi}:

\begin{theorem}\label{th:morita}
	The 2-form $\varpi_\D\in \Omega^2(\DC)$ defines a Morita equivalence of quasi-symplectic groupoids between $(\G_1,\omega_1)$ and $(\G_2,\omega_2)$. 
\end{theorem}

\subsection{Hamiltonian spaces}
As a consequence of Theorem \ref{th:morita}, one obtains a 1-1 correspondence between Hamiltonian spaces for the quasi-symplectic groupoids $(\G_i,\omega_i),\ i=1,2$. 
Since $\G_2$ is finite-dimensional, we are mainly interested in its finite-dimensional Hamiltonian spaces. The  Hamiltonian spaces for 
$(\G_2,\omega_2)$ 
are  $\PSL(2,\R)$-manifolds $M$, equipped with an invariant 2-form $\omega_M$ 
and an equivariant map $\Phi_M\colon M\to \wt{\SL}(2,\R)_+$ satisfying $\d\omega_M=-(\Phi_M)^*\eta$ and the  momentum map condition 
\[  \iota(X_M)\omega_M=-(\Phi_M)^*\on{tr}\big((\theta^L+\theta^R)X\big)\]
along with the minimal degeneracy condition $\ker(\omega_M)\cap \ker T\Phi_M=0$. 
Consider on the other hand the Hamiltonian spaces for $(\G_1,\omega_1)$.  Due to the quotient map $\Diff_\Z(\wt\CC)\ltimes \vir^*_1(\CC)=\wt{\G}_1\to \G_1$, these are in particular
Hamiltonian spaces for $\wt{\G}_1$. The latter are the \emph{Hamiltonian Virasoro spaces}: 
infinite-dimensional 
$\Diff_\Z(\wt\CC)\cong \wt{\Diff}_+(\CC)$-manifolds $\M$ with invariant closed 2-forms $\omega_\M$, such that $\omega_\M$ is (weakly) nondegenerate and satisfies the moment map condition 
\[ \iota(v_\M)\omega_\M=-(\Phi_\M)^*\int_C (\d L)v.\] 
These data define a Hamiltonian space for $\G_1=\wt{\G}_1/\ca{K}$ (see \eqref{eq:ksubgroupoid}) if and only if the subgroupoid $\ca{K}$ acts trivially. In other words, the stabilizers of points $u\in \M$ must satisfy 
\[ \Diff_\Z(\wt\CC)_u\supseteq \ca{K}_{\Phi_\M(u)}.\]
The condition that the associated space $M=(\M\times_{\vir^*_1(\CC)} \DC)/\Diff_\Z(\wt\CC)$ be finite-dimensional amounts to an assumption that $T_u\Phi_\M$ has finite dimensional kernel everywhere; as explained in the appendix (see \eqref{eq:kernelidentification}) 
this kernel is identified with $\ker(T_m\Phi_M)$, where $m$ is a point corresponding to $u$.  
\begin{remark}
	$\dim(\ker T_u\Phi_\M)<\infty$ if and only if  $\M$ is `transversally finite-dimensional', in the sense that  
	the $ \Diff_\Z(\wt\CC)$-orbits have finite codimension. Indeed, the correspondence identifies the normal spaces to the orbits through $u$ and $m$. 
\end{remark}

\begin{remark}
Following Remark \ref{rem:hyperbolic}, the situation simplifies when $\G_2$ is restricted to the open subset 	
$\on{Hyp}_0\subset \wt{\SL}(2,\R)$ of hyperbolic elements of translation number $0$, and $\G_1$ to the corresponding subset 
$\ca{R}_0\subset \vir^*_1(\CC)$. The Hamiltonian $\G_2|_{\on{Hyp}_0}$-spaces are just the Hamiltonian $\G_2$-spaces with $\Phi_M(M)\subset \on{Hyp}_0$. Since the quotient map $\wt{\SL}(2,\R)\to \PSL(2,\R)$ restricts to a diffeomorphism from $\on{Hyp}_0$
to the hyperbolic elements of $\PSL(2,\R)$, these can be further viewed as q-Hamiltonian $\PSL(2,\R)$-spaces with $\PSL(2,\R)$-valued moment map taking values in hyperbolic elements. On the other hand,  since $\ca{K}$ is trivial over $\ca{R}_0$, the Hamiltonian $\G_1|_{\ca{R}_0}$-spaces 
are identified with Hamiltonian Virasoro spaces whose moment map take values in $\ca{R}_0$, with no condition on stabilizers (but still assuming that $T_u\Phi$ has finite dimensional kernel everywhere). Many examples of this type may be constructed from $\PSL(2,\R)$-representation varieties. 
\end{remark}

\subsection{The Maurer-Cartan form on $\on{Diff}_+(\CC)$}
The infinite-dimensional group $\Diff_+(\CC)$ 
 has $\Vect(\CC)=|\Omega|^{-1}_\CC$ as its Lie algebra, and hence carries a left-invariant Maurer-Cartan form 
\[ \theta^L_{\Diff_+(\CC)}\in \Omega^1\big(\Diff_+(\CC),|\Omega|^{-1}_\CC\big).\]
By definition, this is 
characterized by left invariance together with the property $\iota(v)\theta^L_{\Diff_+(\CC)}|_{\on{id}}=v$ 
for all $v\in \Vect(\CC)$. 

To give a concrete formula, choose a parametrization $\CC\cong S^1=\R/\Z$. 
The evaluation map 
\[ \on{Diff}_\Z(S^1)\to S^1=\R/\Z,\ \FF\mapsto \FF(x)\] 
will be denoted  simply by $\FF(x)$; thus $\d\,\FF(x)$ is an ordinary 1-form on $\on{Diff}(S^1)$. Letting $x$ vary, this defines 
\[ \d\FF\in \Omega^1(\Diff_+(S^1),|\Omega|^0_{S^1}).\]
Similarly, $x\mapsto \f{\p \FF}{\p x}\ \partial x$ is a positive density. Similar to our conventions in \ref{subsec:coordinate}, we denote this density by 
\[ \FF'\in  \Omega^0(\Diff_+(S^1),|\Omega|^1_{S^1}).\]

\begin{lemma}
	The left-invariant Maurer-Cartan form on $\on{Diff}_+(S^1)$ is given by 
	\begin{equation}\label{eq:MCformula}
	 -\f{\d \FF}{\FF'}\in \Omega^1(\Diff_+(S^1),|\Omega|^{-1}_{S^1}).
	 \end{equation}
\end{lemma}
\begin{proof}
Let $v=f(x)\f{\p}{\p x}$ be a given vector field on $S^1$, and denote by $t\mapsto \exp(tv)\in \Diff_+(S^1)$ its flow. The resulting left-invariant vector field on the diffeomorphism group acts by right translations, $\FF\mapsto \FF\circ \exp(-tv)$. Hence 
\[ 	\iota(v^L)\d\FF(x)=\f{\p}{\p t}\big|_{t=0}\FF(\exp(-tv)(x))=-\f{\p \FF(x)}{\p x}\ f(x).\qedhere\]
\end{proof}

\begin{remark}
We may also directly  verify that \eqref{eq:MCformula} is left-invariant, and  
satisfies the Maurer-Cartan equation 
$\d \theta^L+\hh[\theta^L,\theta^L]=0$:
\[ \d\Big(-\f{\d \FF}{\FF'}\Big)=\f{1}{(\FF')^2}\d \FF'\wedge \d \FF
=-\Big(-\f{\d \FF}{\FF'}\Big) \wedge \Big(-\f{\d \FF}{\FF'}\Big)'.\]
\end{remark}

Using the Maurer-Cartan form, we may write down the 2-form of the symplectic groupoid 
$\wt{\G}_1=\Diff_\Z(\R)\ltimes \vir^*_\mathsf{1}(S^1)\rra \vir^*_\mathsf{1}(S^1)$.  
Recall that the choice of parametrization defines a splitting of the central extension, 
 describing the latter in terms of the Gelfand-Fuchs cocycle \eqref{eq:GF}.  
The splitting also identifies $\vir^*_1(S^1)$ with the space of quadratic differentials (Hill potentials). 
We write 
\[ \T\in \Omega^0(\vir^*_1(S^1),|\Omega|^2_{S^1})\]
for the map taking an element of $\vir^*_1(S^1)$ to the corresponding Hill potential, thus  $\d\T$ is a 1-form on $\vir^*_1(S^1)$ with values in quadratic differentials. As a special case of \eqref{eq:omegaformula}, we obtain
 \begin{equation}\label{eq:symplecticgroupoidvirasoro}
  \omega_{can}=
 \int_{S^1}\d \T \wedge \f{\d \FF}{\FF'} +\int_{S^1}\T\  \f{\d \FF}{\FF'}  \wedge \Big(\f{\d \FF}{\FF'}\Big)'
 +\f{1}{2}\int_{S^1} \Big(\f{\d \FF}{\FF'}\Big)'''\wedge \f{\d \FF}{\FF'}.
 \end{equation}

\subsection{The distinguished 1-form  on $\DC$}
We return to the general setting, not using a choice of parametrization of $\CC=\wt\CC/\Z$.  There is a distinguished 1-form 
\begin{equation}\label{eq:Theta} 
\Theta\in \Omega^1\big(\DC,   \,   |\Omega|^{ -1}_{\wt\CC} \big)\end{equation}
whose evaluation on tangent vectors $V$ at  $\gamma\in \DC$ is given by \emph{minus} the image of $V$ 
under 
\begin{equation}\label{eq:identifications} T_\gamma \DC\hra 
 T_\gamma(C^\infty(\wt\CC,\RP(1))= 
\Gamma(\gamma^*T\RP(1))\cong \Gamma( T\wt\CC)=\Vect(\wt{\CC})=|\Omega|^{ -1}_{\wt\CC}.\end{equation}
Here we used that 
$\gamma^*T\RP(1)\cong T\wt\CC$, since $\gamma$ is a local diffeomorphism. 
For an explicit formula, write elements of $\DC$ as $\gamma=(\sin\phi:\cos\phi)$.   We may think of $\phi$ as a function on $\DC$ with values in the vector space of 
$\R/\pi\Z$-valued functions on $\wt{\CC}$. Taking its exterior derivative gives a 1-form 
\[  \d\phi\in \Omega^1\big(\DC,\,   |\Omega|^{0}_{\wt\CC} \big).\]
On the other hand, we may apply $\partial\colon  |\Omega|^{0}_{\wt\CC} \to 
|\Omega|^{1}_{\wt\CC}$  to obtain 
\[\phi'=\partial\phi\in \Omega^0\big(\DC,    |\Omega|^{1}_{\wt{\CC}} \big).\]
\begin{proposition}
	We have 
	\begin{equation}\label{eq:theta} \Theta=-\dfrac{\d\phi}{\phi'}.
	\end{equation}
In terms of the normalized lifts (Definition \ref{def:normalized}) 
	\begin{equation}\label{eq:thetanormalizedlifts}
	\Theta=u_1\d u_2-u_2 \d u_1.\end{equation}
\end{proposition}
(In \eqref{eq:thetanormalizedlifts}, we are treating $u_i$ 
 as a $|\Omega|^{-\f{1}{2}}_{\wt{\CC}}$-valued functions on $\DC$. Strictly, the normalized lifts are only defined up to an overall sign; 
 thus $u_i$ are defined on a 	double cover $\wh{\D}(\CC)$.)
\begin{proof}
	To verify \eqref{eq:theta}, choose a local coordinate $t$ on $\wt\CC$, and let $V=f(t)\f{\p}{\p \phi}$ be a vector field along $\gamma$. The corresponding vector field 
	on $\wt{\CC}$ is 
	\[ f(t)\f{\p}{\p t}=\f{1}{\phi'(t)} f(t)\f{\p}{\p \phi}=-\iota_V \Theta.\]  
	The alternative expression \eqref{eq:thetanormalizedlifts} reduces to 
	\eqref{eq:theta} by using the  formula \eqref{eq:normalized} for the normalized lifts. 
\end{proof}
Note the similarity of \eqref{eq:theta} with the expression for the left-invariant Maurer-Cartan form on $\Diff_+(S^1)$. 
For the following result, let
\[ \Upsilon\colon \R^2\to \mf{sl}(2,\R),\ \ \mathbf{s}\mapsto  \hh \mathbf{s} (J_1 \mathbf{s})^\top
\]
with $J_1$ as in \eqref{eq:j} and with 
\[ \mathbf{s}=\left(\begin{array}{c} s_1 \\ s_2\end{array}\right).\] 
Up to a factor, this is is the momentum map for the $\SL(2,\R)$-action on the symplectic vector space $(\R^2,\,\d s_1\wedge\d s_2)$. 
In terms of the $\SL(2,\R)$-invariant 1-form 
$s_1\d s_2-s_2\d s_1$, it satisfies
\begin{equation}\label{eq:XS}
\iota(X_{\R^2})(s_1\d s_2-s_2\d s_1)=2\on{tr}(X\,\Upsilon(\mathbf{s}))\end{equation}
for $X\in \g$. Given $u_i\in |\Omega|^{-\f{1}{2}}_{\wt{\CC}}$, we may form substitute 
\[ \mathbf{u}=\left(\begin{array}{c} u_1 \\ u_2\end{array}\right)\] 
for $\mathbf{s}$, defining a vector field 
\[ \on{tr}(X\,\Upsilon(\mathbf{u}))\in |\Omega|^{-1}_{\wt\CC}=\Vect(\wt{\CC}).\]
See Remark \ref{rem:vectorfields} below for the relevance of these vector fields.

\begin{proposition}[Properties of the form $\Theta$]\label{prop:propertiestheta} \phantom{.}
		\begin{enumerate}
		\item  For $g\in \PSL(2,\R)$ and  $\FF\in \Diff_\Z(\wt\CC)$,
		\[g^*\Theta=\Theta,\ \ \ \  \FF^*\Theta=\Ad_\FF\Theta.\]
		\item\label{it:c1}  $\Theta$ satisfies the Maurer-Cartan equation 
		\[ \d\Theta+\hh [\Theta,\Theta]=0\]
		using the Lie bracket on $|\Omega|_{\wt\CC}^{-1}=\on{Vect}(\wt\CC)$. 
		\item\label{it:cb} 
		The contractions with the generating vector fields for the $\Diff_\Z(\wt\CC)$-action 
			on $\DC$ are given by 
		\[\iota(v_\D)\Theta=v\]
		for $v\in \Vect(\CC)$.	
		\item\label{it:cd} 
		The contractions with the generating vector fields for the $\PSL(2,\R)$-action 
		on $\DC$ are given by 
		\[\iota(X_\D)\Theta=2\on{tr}(X \Upsilon(\mathbf{u}))\]
		for $X\in \g$. 
\item\label{it:ce} Under the $\Z$-action on $|\Omega|^{-1}_{\wt\CC}$, 
			 \[\kappa^*\Theta-\Theta=-2\on{tr}\left(
			q^*\theta^L\ \Upsilon(\mathbf{u}))
			\right) .\]
\item\label{it:cf} The image of $\Theta$ under the `covariant derivative' $D_L\colon  |\Omega|^{-1}_{\wt\CC} \to  |\Omega|^{2}_{\wt\CC} $ 
(see \eqref{eq:dv}) is given by 
\[
	 D_L\Theta=p^*\d L
\]
where  
$ \d L\in \Omega^1(\CC,|\Omega|^{2}_\CC )$
is the tautological 1-form of the affine space $\vir^*_1(\CC)=\on{Hill}(\CC)$. In particular, $D_L\Theta$ is $\Z$-invariant. 
	\end{enumerate}
\end{proposition}
\begin{proof}
\begin{enumerate}
	\item Equivariance under the action of $\Diff_\Z(\wt\CC)$ is clear from the coordinate-free definition. The $\PSL(2,\R)$-invariance follows from 
	\eqref{eq:thetanormalizedlifts}, since the 1-form $s_1\d s_2-s_2\d s_1\in \Omega^1(\R^2)$ is $\SL(2,\R)$-invariant. (It is the contraction 
	of the symplectic form $\d s_1\wedge \d s_2$ by the Euler vector field.) 
	\item 
	We may assume $\CC=S^1$, so that $\hh [\Theta,\Theta]=\Theta\wedge \Theta'$. 
	The property $u_2'u_1-u_1'u_2=-1$ gives $u_1\d u_2'-u_2\d u_1'+u_2' \d u_1-u_1'\d u_2=0$, and therefore	
	\begin{align*}
	\Theta\wedge \Theta'
	&=(u_1\d u_2-u_2\d u_1)\wedge (u_1'\d u_2-u_2'\d u_1+u_1\d u_2'-u_2\d u_1')\\
	&=2 (u_1\d u_2-u_2\d u_1)\wedge  (u_1'\d u_2-u_2'\d u_1)\\
	&=2(u_1u_2'-u_2u_1')\d u_1\wedge \d u_2\\
	&=-2 \d u_1\wedge \d u_2=-\d\Theta.
	\end{align*}	
	\item  This follows since  $\gamma_*(v)=-v_\D|_\gamma$, as vector fields along $\gamma$. 	
	(The minus sign arises since the $\Diff_\Z(\wt\CC)$-action on $\DC$ comes from the action on the \emph{source} of 
	$\gamma\colon \wt{\CC}\to \RP(1)$.)
	\item We regard $\mathbf{u}=(u_1\ u_2)^\top$ as a function on $\DC$ (more precisely, its double cover	
	$\wh{\D}(\CC)$) 
	with values in $\R^2\otimes |\Omega|^{-1/2}_{\wt\CC}$. This function is $\SL(2,\R)$-equivariant, using the standard action on $\R^2$. 
	Using the formula \eqref{eq:thetanormalizedlifts} for $\Theta$, the expression for $\iota(X_\D)\Theta$ is  a consequence of the formula 
	\eqref{eq:XS} for $\iota(X_{\R^2})(s_1\d s_2-s_2\d s_1)$. 
	\item Formula \eqref{eq:thetanormalizedlifts} may be written as $\Theta=\on{tr} \big( (\d\mathbf{u})\mathbf{u}^\top J_1\big)$
	The normalized lifts transform according to 
	\[ \kappa^*\mathbf{u}=A\mathbf{u}\] 
	where the matrix valued function 
	$A\colon \DC\to \SL(2,\R)$ is the composition of $q$ with the covering map $\wt{\SL}(2,\R)\to \SL(2,\R)$. 
	Taking the exterior differential of this equation, we obtain 
	\[ \kappa^*\d\mathbf{u}=A\big(\d\mathbf{u}+(q^*\theta^L)\, \mathbf{u}\big).\]
Hence, 
	\begin{align*}
		\kappa^*\Theta &= \on{tr}\left(A\big(\d\mathbf{u}+(q^*\theta^L)\, \mathbf{u}\big) \mathbf{u}^\top A^\top  J_1\right)\\
		&=\on{tr}\left(\big(\d\mathbf{u}+(q^*\theta^L)\, \mathbf{u}\big) \mathbf{u}^\top A^\top  J_1 A\right)\\
		&=	\on{tr}\left(\big(\d\mathbf{u}+(q^*\theta^L)\, \mathbf{u}\big) \mathbf{u}^\top J_1\right)\\
		&=\Theta-2\on{tr}\left(q^*\theta^L\ \Upsilon(\mathbf{u}))\right).
	\end{align*}
	\item 
		We may assume $\CC=S^1=\R/\Z$, and $Lu=u''+\T u$ where $\T=(\phi')^2+\hh\S(\phi)$. Regarding $\T$  as an element of 
	$\Omega^0(\D_+(S^1),|\Omega|^2_\CC)$, as above,  this says $p^*\d L=\d \T$.  
	From the formula for the Schwarzian derivative $\S(\phi)=(\phi')^{-1}\phi'''-\f{3}{2}(\phi')^{-2}(\phi'')^2$, we obtain, through direct but lengthy calculation, 
	\[ 		\d \S(\phi)=-\big(\S(\phi)'\,\Theta+2\S(\phi)\,\Theta'+\Theta'''\big).\]
	On the other hand, using $\Theta'=-(\phi')^{-1}\d\phi'-(\phi')^{-1}\phi''\Theta$,
	\[ \d \big((\phi')^2\big)= -\Big(\big((\phi')^2\big)' \Theta+2 \big((\phi')^2\big)\Theta'\Big).\]
	Hence \[\d\T=\d\big((\phi')^2+\hh\S(\phi)\big)=-\big(\T'\Theta+2\T\Theta'+\hh\Theta''\big)=D_L\Theta.\qedhere\]
\end{enumerate}	
\end{proof}

\begin{remark}\label{rem:vectorfields}
	For $\gamma\in \DC$, the map   
	\begin{equation}\label{eq:varrhogamma}  \varrho_\gamma\colon \g\to |\Omega|^{-1}_{\wt\CC}=\Vect(\wt\CC),\ \ 
	X\mapsto -\iota(X_\D)\Theta|_\gamma
	\ \ \ \end{equation}
	is a Lie algebra action on $\wt\CC$: it is bracket preserving as a consequence of the Maurer-Cartan equation for $\Theta$. 
	From the definition of $\Theta$, one finds 
	that $\varrho_\gamma(X)$ is simply the  pullback of the generating vector field $X_{\RP(1)}$
	under the local diffeomorphism $\gamma\colon \wt{\CC}\to \RP(1)$. 
	Under the $\Z$-action,
	%
		\[ \kappa^*\big(\varrho_\gamma(X)\big)=-2\on{tr}( X\Ad_{q(\gamma)}\Upsilon(\mathbf{u}))
	=-2\on{tr}\big( (\Ad_{q(\gamma)^{-1}}X)\,\Upsilon(\mathbf{u})\big)=\varrho_\gamma(\Ad_{q(\gamma)^{-1}}X)
	\]
	where we used that $\Upsilon$ is $\SL(2,\R)$-equivariant.
	In particular, the vector field $\varrho_\gamma(X)$ is periodic (and so descends to $\CC$) if and only if 
	$X$ is fixed under the adjoint action of the monodromy.  
	These vector fields are exactly the elements of $\Vect(\CC)_{p(\gamma)}$,  
	the Lie algebra of the stabilizer $\Diff_+(\CC)_{p(\gamma)}$. See \cite{kir:orb} and \cite[Section 1.3]{ovs:pro}. 
\end{remark}

\subsection{The 2-form}
We now have all the ingredients to describe the 2-form on the space of developing maps.  The formula involves an integral of
the density valued 2-form 
\[ (D_L\Theta)\, \Theta\in \Omega^2(\DC,|\Omega|^1_{\wt{\CC}})\]
 over the segment from $x_0$ to $\kappa(x_0)$ for some choice of $x_0\in\wt{\CC}$. The result becomes independent of $x_0$, 
 by adding a boundary term involving the 
 bilinear concomitant $B_{L,x_0}$, see \eqref{eq:bl}.
 \begin{proposition}\label{prop:independence}
 The 2-form $\varpi_{\D}\in \Omega^2(\DC)$, given as 
 \[ \varpi_\D=-\int_{x_0}^{\kappa(x_0)}(D_L\Theta)\, \Theta-B_{L,x_0}(\Theta,\kappa^*\Theta)\]
 does not depend on the choice of $x_0\in\wt\CC$.  
 \end{proposition}
 \begin{proof}
 	Let $x_1\neq x_0$ be another choice. Without loss of generality, assume $x_1>x_0$ with respect to the orientation.  We have 
 	\begin{align*} 
 	\int_{x_0}^{\kappa(x_0)} (D_L\Theta)\Theta-  \int_{x_1}^{\kappa(x_1)} (D_L\Theta)\Theta&=
 	\int_{x_0}^{x_1}(D_L\Theta)\Theta-\int_{\kappa(x_0)}^{\kappa(x_1)}(D_L\Theta)\Theta
 	\\&=\int_{x_0}^{x_1}\big((D_L\Theta)\Theta-\kappa^*(D_L\Theta)\kappa^*\Theta\big)
 	\\&=-\int_{x_0}^{x_1} (D_L\Theta)\, (\kappa^*\Theta-\Theta)
 	\end{align*}
 	where we used that $D_L\Theta$ is $\Z$-invariant (Proposition \ref{prop:propertiestheta}). Using the integration by parts formula
 	\eqref{eq:intbyparts} with 
 	\[ \int_{x_0}^{x_1} \Theta\,D_L(\kappa^*\Theta-\Theta)=0\]
 	(since $D_L\kappa^*\Theta=\kappa^*D_L\Theta=D_L\Theta$) the last expression becomes
 \[ -B_{L,x_0}(\Theta,\kappa^*\Theta-\Theta)+B_{L,x_1}(\Theta,\kappa^*\Theta-\Theta).\]	
 Finally, note that $B_{L,x_i}(\Theta,\Theta)=0$ by symmetry. 	 
 \end{proof}	
 	
 Using the explicit expression for the boundary terms, as given in \eqref{eq:bdryterm}, we can write down the formula 
 for the case $\CC=S^1=\R/\Z$, as follows:
\begin{equation}\label{eq:varpis1}
\varpi_{\D_+(\S^1)}=-\int_0^1 (D_L\Theta)\, \Theta-\Big(
2\T_0\Theta_0\Theta_1+\hh \Theta_0\Theta_1''-\hh \Theta_0'\Theta_1'+\hh \Theta_0''\Theta_1
\Big).
\end{equation}
Here $\T$ is the Hill potential corresponding to $L=p(\gamma)$, and  
$\T_x,\Theta_x,\Theta_x',\Theta_x''$ are the evaluations at $x\in \R$. 
Writing everything in terms of 
$\T=(\phi')^2+\f{1}{2}\phi'''/\phi'-\f{3}{4}(\phi''/\phi')^2$ and $\Theta=-\d\phi/\phi'$, this becomes a highly complicated expression in terms of $\phi$ and its first three derivatives. 

The properties of $\varpi_\D$, as described in Theorem \ref{th:propertiesofvarpi} may now be proved directly, using the properties of the 1-form $\Theta$. On the other hand, they will also be immediate consequences of the construction by Drinfeld-Sokolov reduction, described in the next Section.

\section{Construction via Drinfeld-Sokolov method}\label{sec:ds}
We shall now give a construction of the 2-form $\varpi_\D$, using the Drinfeld-Sokolov approach \cite{dri:kdv}. Throughout, we take $\CC$ to be a compact, connected, oriented 1-manifold, presented as $\CC=\wt\CC/\Z$ for a universal cover $\wt{\CC}$.
The quotient map will be denoted $\pi\colon \wt{\CC}\to \CC$. 

\subsection{Principal bundles over $\CC$} \label{subsec:prin}
Throughout this subsection, 
 $G$ will denote an arbitrary connected Lie group; later we will specialize to $G=\PSL(2,\R)$. 
Let $Q\to \CC$ be a principal $G$-bundle, and let $\on{Gau}(Q)$ be its group of gauge transformations, with Lie algebra 
\[ \gau(Q)=\Omega^0(\CC,Q\times_G\g)\] the space 
of sections of the adjoint bundle. Let $\A(Q)$ be the affine space of  principal connections on $Q$; the underlying linear space is $\Omega^1(\CC,Q\times_G\g)$. 
The group $\Gau(Q)$ acts on $\A(Q)$ by gauge transformations of connections, with generating vector fields $A\mapsto -\partial_A\xi$ where 
\[ \partial_A\colon \Omega^0(\CC,Q\times_G\g)\to \Omega^1(\CC,Q\times_G\g)\] 
is the covariant derivative.

 The action groupoid for the gauge action is Morita equivalent to the transformation groupoid for the conjugation action of $G$ on itself
\begin{equation}\label{eq:corrLG3}
\begin{tikzcd}
[column sep={7em,between origins},
row sep={4.5em,between origins},]
\Gau(Q)\ltimes \A(Q) \arrow[d,shift left]\arrow[d,shift right]
& \P(Q) \arrow[dl, "p"] \arrow [dr, "q"']   & G\ltimes G \arrow[d,shift left]\arrow[d,shift right] \\
\A(Q) & & G
\end{tikzcd}
\end{equation}
Here the equivalence bimodule is the space 
\[ \P(Q)=\{\tau\in \Gamma(\pi^*Q)|\ \exists h\in G\colon 
\kappa^*\tau=h\cdot \tau\}\]
of quasi-periodic sections of the pullback bundle $\pi^*Q\to \wt{\CC}$. Furthermore, $q(\tau)=h$ is the \emph{monodromy}, and $p(\tau)$ is the unique principal connection such that $\tau$ is horizontal (for the pullback connection). 
Since we are dealing with action groupoids, the Morita equivalence just means that the two actions commute, and that they are both principal actions, with $p$ and $q$ as their respective quotient maps. In particular, \eqref{eq:corrLG3}
gives a 1-1 correspondence between $\Gau(Q)$-orbits in $\A(Q)$ and conjugacy classes in $G$. 

\begin{remark}
	\begin{enumerate}
		\item 	A choice of parametrization $\CC\cong S^1$  and trivialization $Q\cong S^1\times G$	identifies $\Gau(Q)=LG$, with the affine action on 
		$\A(Q)=\Omega^1(S^1,\g)$. 		
		\item If $G$ is not simply connected, then the group $\Gau(Q)\cong LG$ will be disconnected, and one may prefer to work with its identity component $\Gau_0(Q)\cong L\wt{G}$ (where $\wt{G}$ is the universal cover of $G$). The quotient map under the action of $\Gau_0(Q)$ defines a lifted monodromy map 
		$\P(Q)\to \wt{G}$, and gives a Morita equivalence 
	\begin{equation}\label{eq:corrLG4}	\begin{tikzcd}
		[column sep={7em,between origins},
		row sep={4.5em,between origins},]
		\Gau_0(Q)\ltimes \A(Q) \arrow[d,shift left]\arrow[d,shift right]
		& \P(Q) \arrow[dl, "p"] \arrow [dr, "q"']   & G\ltimes \wt{G} \arrow[d,shift left]\arrow[d,shift right] \\
		\A(Q) & & \wt{G}
		\end{tikzcd}
\end{equation}

\item One may also allow for disconnected $G$. Here $Q$ is possibly non-trivial, and the image of the map $q$ is a union of components of $G$, 
depending on the topological type of $Q$.  
This generalization is relevant for the setting of twisted loop groups.   		\end{enumerate}
\end{remark}

The gauge action of $\Gau(Q)$ on the space of connections extends to the larger group $\Aut(Q)$ of all principal bundle automorphisms. Let $\Aut_+(Q)$ be the subgroup of automorphisms 
such that the induced action on the base $\CC$ preserves orientation. It fits into an exact sequence 
\begin{equation}\label{eq:gaugesequence} 1\lra  \on{Gau}(Q)\lra \Aut_+(Q)\lra \Diff_+(\CC)\to 1,\end{equation}
with a lift
\begin{equation}\label{eq:gaugesequence1} 1\lra  \on{Gau}(Q)\lra  \Aut_\Z(\pi^*Q) \lra \Diff_\Z(\wt\CC)\to 1\end{equation}
where $\Aut_\Z(\pi^*Q)$ are the $\Z$-equivariant automorphisms of $\pi^*Q$. Note that the action of this group 
on $\P(Q)$ commutes with the action of $G$.  

We shall need a certain `tautological' 1-form  
\begin{equation}\label{eq:xi} \Xi\in \Omega^1(\P(Q), |\Omega|^{0}_{\wt\CC} \otimes \g).
\end{equation}
It may be described in terms of its values $\Xi_x\in \Omega^1(\P(Q),\g)$. Let
\[ \on{ev}_x\colon \P(Q)\to \pi^*Q_x,\ \tau\mapsto \tau(x).\] 
Since $Q_{[x]}$ is a principal $G$-bundle over a point, it has a unique 
connection 1-form $\theta_{[x]}\in \Omega^1(Q_{[x]},\g)$. Put $\Xi_x=\on{ev}_x^*\theta_{[x]}$. 
%

\begin{proposition}[Properties of $\Xi$]\label{prop:propertiesxi}\phantom{.}
	\begin{enumerate}
		\item For $g\in G$ and $\FF_Q\in \Aut_\Z(\pi^*Q)$, with base map $\FF\in \Diff_\Z(\wt{\CC})$
		\[ g^*\Xi=\Ad_g\Xi,\ \ (\FF_Q)^*\Xi=\FF\cdot \Xi.\]
		\item $\Xi$ satisfies the Maurer-Cartan equation, $\d\Xi+\hh [\Xi,\Xi]=0$.
		\item For $\xi\in \gau(Q)$, $X\in \g$, $\tau\in \P(Q)$,
		\[ \iota(\xi_{\P})\Xi|_\tau=-\tau^*\wt{\xi},\ \ \iota_{X_{\P}}\Xi=X\]
		(where $\xi\in \Gamma(Q\times_G \g)$ is identified with an equivariant map $Q\to \g$, and $\wt{\xi}$ is its pullback to a map on $\pi^*Q$)
		\item Under the $\Z$-action on $|\Omega|^{0}_{\wt\CC}$
			\begin{equation}\label{eq:xik} \kappa^*\Xi=\Ad_{q}\Xi-q^*\theta^R.\end{equation}
	\end{enumerate}
\end{proposition}
\begin{proof}
Properties (a)--(c) are simple consequences of the definition.  Consider the transformation property under 
$\kappa$. 
At any given $x\in \wt{\CC}$, this says that 
\begin{equation}\label{eq:xik1}
 \Xi_{\kappa(x)}=\Ad_{q}\Xi_x-q^*\theta^R.\end{equation}
By definition, elements of $\P(Q)$ satisfy $\tau(\kappa(x))=q(\tau)\cdot \tau(x)$. That is, $\on{ev}_{\kappa(x)}$ is 
given by $\on{ev}_x$ followed by the gauge action of $q(\tau)$ on $Q_{\pi(x)}$. 
Choose a trivialization $Q_{\pi(x)}\cong G$, so that $\theta_x=\theta^L$, and 
the gauge action is multiplication from the right by $q(\tau)^{-1}$.
Equation \eqref{eq:xik1} follows from the property of 
Maurer-Cartan forms. 
\end{proof}

 \subsection{The 2-form $\varpi_{\P}$}\label{subsec:varpip}
 Suppose now that $\g=\on{Lie}(G)$ comes equipped with a non-degenerate invariant symmetric bilinear form $\l\cdot,\cdot\r\colon\g\times\g\to \R$, henceforth referred to as a metric.  
The metric, together with integration, defines a non-degenerate $\Aut_+(Q)$-invariant 
pairing between $\Omega^1(\CC,Q\times_G\g)$ and $\gau(Q)$. We may hence regard $\Omega^1(\CC,Q\times_G\g)$ as the smooth dual of $\gau(Q)$, with the $\Gau(Q)$-action as the coadjoint action. It is the linear action underlying the affine $\Gau(Q)$-action on the affine space $\A(Q)$. 
By the same mechanism as explained at the beginning of section \ref{subsec:vir}, 
this defines a central extension
\begin{equation}\label{eq:central} 0\to \R\to \wh{\gau}(Q)\to \gau(Q)\to 0.\end{equation}
As a vector space, $\wh{\gau}(Q)$ consists of all affine-linear functionals $\wh{\xi}\colon \A(Q)\to \R$ for which the underlying linear functional on $\Omega^1(\CC,Q\times_G\g)$ is given by an element $\xi\in \gau(Q)$; 
the bracket reads as
\[ [\wh{\xi}_1,\wh{\xi}_2](A)=\int_\CC \l \partial_A\xi_1,\, \xi_2\r.\]
From now on, we consider the affine action of $\Gau(Q)$ on $\A(Q)=\wh{\gau}(Q)^*_1$ as the \emph{coadjoint gauge action}. 

The following result is due to \cite{al:ati} for the case $Q=S^1\times G$; 
see also \cite{al:mom,loi:spi}. 
\begin{theorem}\label{th:varpiforloopgroups}
There is a natural 2-form 
\[ \varpi_{\P}\in \Omega^2(\P(Q))\]  with the following properties: 
\begin{enumerate}
	\item The exterior differential of $\varpi_{\P}$ is 
	\[ \d\varpi_{\P}=q^*\eta,\]
	where $\eta=-\f{1}{12}\l \theta^L,\, [\theta^L,\theta^L]\r$ is the Cartan 3-form on $G$. 
	\item $\varpi_{\P}$  is $G$-invariant, and its contractions with generating vector fields for the $G$-action are 
	\[ \iota_{X_{\P}} \varpi_{\P}=-\hh q^*\l \theta^L+\theta^R,\, X\r.\]
\item $\varpi_{\P}$  is $\Gau(Q)$-invariant, and its contractions with generating vector fields, for $\xi\in \mf{gau}(Q)=\Gamma(Q\times_G\g)$, are 
\[ \iota(\xi_{\P}) \varpi_{\P}=-p^*\l d A,\xi\r .\]
In fact, $\varpi_{\P}$ is $\Aut_+(Q)$-invariant. 
\end{enumerate}
For a given choice of base point ${x}_0\in \wt{\CC}$, we have 
\begin{equation}\label{eq:varpiformula}\varpi_{\P}=-\hh \int_{{x}_0}^{\kappa({x}_0)} \l \Xi,\,  \Xi'\r+\hh\, \l\Xi_{{x}_0},\,q^*\theta^L\r. \end{equation}

\end{theorem}

Here naturality means that $\varpi_{\P}$ does not depend on any additional choices. In particular; it 
does not depend on the choice of $x_0$.

\begin{proof}
The $G$-invariance is clear from \eqref{eq:varpiformula} and the invariance of the metric on $\g$. The independence of 
\eqref{eq:varpiformula} of the choice of $x_0$ is verified by an argument parallel to that for $\Theta$: Replacing $x_0$ with $x_1$ changes the first term by 
\begin{align*} \hh \int_{x_0}^{x_1}\l \Xi,\Xi'\r-\hh \int_{\kappa(x_0)}^{\kappa(x_1)}\l \Xi,\Xi'\r
&=\hh \int_{x_0}^{x_1}\l \Xi,\Xi'\r-\hh \int_{x_0}^{x_1}\l \Ad_q\Xi-q^*\theta^R,\Ad_q\Xi'\r\\
&=\hh\l q^*\theta^L,\Xi\r\Big|_{x_0}^{x_1},
\end{align*}
which exactly compensates the change of the boundary term. This also shows that $\varpi_{\D}$ is natural: Given principal $G$-bundles $Q_i\to \CC_i=\wt{\CC}_i/\Z$ for $i=1,2$, and a $\Z$-equivariant principal bundle morphism $\pi_1^*Q_1\to \pi_2^*Q_2$ between their lifts, the 2-forms $\varpi_{\P(Q_1)}$ and $\varpi_{\P(Q_2)}$ are related by pullback.  In particular, $\varpi_{\P}$ is $\Aut_+(Q)$-invariant. 
The remaining claims are proved in \cite{al:ati} (see also \cite[Appendix C]{loi:spi}) for the case $Q=S^1\times G$, hence we just 
need to make the relevant identifications for this case. We have  $\Gau(Q)=C^\infty(S^1,G)=LG$, while $\P(Q)$ can be identified with paths $\gamma\colon \R\to G$ satisfying $\gamma(x+1)=q(\gamma)\gamma(x)$, by the map taking $\gamma$ to the section $\tau(x)=(x,\gamma(x)^{-1})$ of $\pi^*Q=\R\times G$. Writing $\gamma_x\colon \P(Q)\to G$ for the map taking $\tau$ to  $\gamma(x)$ we obtain 
\[ \Xi_x=(\gamma_x^{-1})^*\theta^L=-\gamma_x^*\theta^R.\]
Using $q(\gamma)^*\theta^L=\Ad_{\gamma_0}\gamma_1^*\theta^L-\gamma_0^*\theta^R$, and putting $x_0=0$, we have  
\[ \varpi_{\P}=-\hh\int_0^1 \l \gamma_x^*\theta^R,\,\f{\p}{\p x} \gamma_x^*\theta^R\r\,\delta x
-\hh \l \gamma_0^*\theta^L,\,\gamma_1^*\theta^L\r\]
which is the expression used in \cite{al:ati} (except for a different sign convention).
\end{proof}


\begin{corollary}
	The space $\P(Q)$, equipped with the 2-form $\varpi_{\P}$, defines a Morita equivalence between the symplectic groupoid 
$\Gau(Q)\ltimes \wh{L\g}^*_1\rra 	\wh{L\g}^*_1$ and the quasi-symplectic groupoid $G\ltimes G\rra G$ (integrating the Cartan-Dirac structure). 
\end{corollary}

\subsection{The Drinfeld-Sokolov embedding}\label{subsec:incl}\bigskip

We shall now specialize to the structure group $G=\PSL(2,\R)$, with the metric on $\g=\mf{sl}(2,\R)$ given by $\l X,\, Y\r=2\on{tr}(XY)$. 
We shall construct a principal $G$-bundle $Q\to \CC$ and an inclusion $\DC\hra \P(Q)$ under which the 2-form $\varpi_{\D}$ 
is the pullback of the corresponding form $\varpi_{\P}$ on $\P(Q)$. 

The following constructions are inspired by  Segal's paper \cite{seg:geo}. 
Since solutions of a given Hill operator are uniquely determined by their 1-jet at any given point, one considers the jet bundle
\begin{equation}\label{eq:E} E=J^1(|\Lambda|_\CC^{-\f{1}{2}}).\end{equation}
This is a rank 2 vector bundle, with a reduction of the structure group to $\SL(2,\R)$, using the trivialization of $\det(E_x)=\wedge^2 E_x$ by 
$j_x(u_1)\wedge \,j_x(u_2)$, for $-\f{1}{2}$-densities $u_1,u_2$ with Wronskian $W(u_1,u_2)=-1$. We take $Q\to \CC$ to be the associated principal $\PSL(2,\R)$-bundle, obtained as the quotient of the $\SL(2,\R)$-frame bundle by the action of the center. 
As before, we denote by $\pi^*Q\to \wt\CC$ its pullback. 

For any developing map $\gamma\in \DC$ we have the normalized lift (Definition \ref{def:normalized}) 
$u_1,u_2$, defined up to an overall sign. At any point $x\in \wt\CC$, this defines an  $\SL(2,\R)$-frame for $E_x$ defined up to an overall sign, i.e, an element of $\pi^*Q_x$. The quasi-periodicity of the normalized lift means that the 
resulting section of $\pi^*Q$ is quasi-periodic. This defines an inclusion  
\begin{equation}\label{eq:iota} \wh{\iota}\colon \DC\to \P(Q).\end{equation}
The map \eqref{eq:iota} is $G$-equivariant, and descends to the 
\emph{Drinfeld-Sokolov embedding}
\begin{equation}\label{eq:descent} \iota\colon \vir_1^*(\CC)=\on{Hill}(\CC)\hra \A(Q).\end{equation} 
\begin{remark}
	The Drinfeld-Sokolov embedding may also be seen as follows. By the existence and uniqueness theorem for second order ODE's, every solution to a Hill operator $L$ is uniquely determined by its 1-jet. Hence, 
	$L$ determines a linear connection $\nabla$ on $E$, with the property that for all local sections $u$ of $|\Lambda|_\CC^{-\f{1}{2}}$, 
	\[ Lu=0\Leftrightarrow \nabla j^1(u)=0.\] 
	In turn, $\nabla$ defines an element of the space $\A(Q)$ of projective connections.  
\end{remark}

The main result of this subsection relates the 2-forms on the space of developing maps and on the space of quasi-periodic sections. 

\begin{theorem}\label{th:pullback} The 2-forms on the space $\DC$ of developing maps and on the space $\P(Q)$ of quasi-periodic sections are related by pullback:
	\[ \varpi_{\D}=\wh{\iota}^*	 \varpi_{\P}.\]
\end{theorem}

In particular, the properties of the 2-form $\varpi_\D$, as listed in Theorem \ref{th:propertiesofvarpi}, follow from the corresponding properties of $\varpi_\P$ (Theorem \ref{th:varpiforloopgroups}). 

\begin{proof}
We shall show separately that the integrands match, i.e. 
\begin{equation}\label{eq:theintegrands} \wh{\iota}^* \Big(\hh\l\Xi,\Xi'\r\Big) =\Theta\,(D_L\Theta),\end{equation}
and the boundary terms (for given choice of $x_0$) match as well. 	For the calculation, we may assume $\CC=S^1=\R/\Z$.  The $r$-densities $|\partial x|^r$ define a  trivialization of the $r$-density bundles $|\Lambda|^r_\CC,\ \ |\Lambda|^r_{\wt\CC}$, which we will use to identify densities with functions. 
 In particular, we obtain a trivialization of $E$, given at $x\in S^1$ by 
 %
 \[E_x=J^1_x(|\Lambda|_\CC^{-\f{1}{2}})\stackrel{\cong}{\lra}\R^2,\ \ \   j^1_x(u)\mapsto \Big(\begin{array}{c}u'(x)\\ u(x)\end{array}\Big).\]
Using square brackets to denote the image of an $\SL(2,\R)$-matrix in $\PSL(2,\R)$, the inclusion $\wh{\iota}$ is given by 
\[ \wh{\iota}\colon \gamma\mapsto \tau=\left[\begin{array}{cc}u_1'& u_2'\\ u_1& u_2
\end{array}\right],\]
where $u_1,u_2$ is a normalized lift of $\gamma$. By definition, $u_1'u_2-u_1u_2'=1$ and $u_i''=-\T u_i$. 
Recall that $\Xi$ is pointwise  the pullback of the left-invariant Maurer Cartan form $\theta^L$. Thus, $\wh{\iota}^*\Xi=\tau^{-1}\d\tau$. 
As it turns out, it will be more convenient to work with 
\[ \check{\Xi}\in \Omega^1(\P(Q),\gau(Q)),\] 
given in terms of the parametrization as $\wh{\iota}^*\check{\Xi}=(\d\tau)\tau^{-1}$. 
The integral term in the formula for $\varpi_{P(Q)}$ can be computed as 
\[ \hh \int_{0}^{1} \l\Xi,\,\Xi'\r=\hh \int_0^1 \l \check{\Xi},\, \partial_A \check{\Xi}\r\] with the covariant derivative $\partial_A \xi =\xi'+[A,\xi]$
corresponding to $A=-\tau'\,\tau^{-1}$. A calculation shows 
\begin{equation}\label{eq:Aconnection} A=-
\left(\begin{array}{cc}u_1'& u_2'\\ u_1& u_2
\end{array}\right)'  \left(\begin{array}{cc}u_2& -u_2'\\ -u_1& u_1'
\end{array}\right)=
\left(\begin{array}{cc} 0& \T\\ -1& 0 \end{array}\right).\end{equation}
Similarly, 
\[ \check{\Xi}=\left(\begin{array}{cc}\d u_1'& \d u_2'\\ \d u_1& \d u_2
\end{array}\right)
\left(\begin{array}{cc}u_2& -u_2'\\ -u_1& u_1'
\end{array}\right)=
\left(\begin{array}{cc} 
-u_1\d u_2'
+u_2\d u_1'
&u_1'\d u_2' -u_2'\d u_1'\\-u_1\d u_2+u_2\d u_1 & u_1'\d u_2-u_2' \d u_1
\end{array}\right)
\]
The matrix entries may be expressed in terms
of 
\begin{align*}
\Theta&=u_1\d u_2-u_2\d u_1\\
\Theta'&=2(u_1'\d u_2-u_2'\d u_1)\\
\Theta''
&
=-2\T \Theta +2(u_1'\d u_2'-u_2'\d u_1')
\end{align*}
(we used $0=\d(u_1u_2'-u_2u_1')=u_2'\d u_1-u_1'\d u_2-u_2\d u_1'+u_1\d u_2'$
and $u_i''=-\T u_i$). The result is
 \[\check{\Xi}=
\left(\begin{array}{cc}-\hh \Theta' & \hh \Theta''+\T \Theta \\ -\Theta  & \hh\Theta' \end{array}\right)\]
For the covariant derivative $\partial_A\check{\Xi}=\check{\Xi}'+[A,\check{\Xi}]$, we find
\[ \partial_A\check{\Xi}
=\left(\begin{array}{cc}0 & \hh \Theta'''+\T '\Theta+2\T\Theta' \\ 0 & 0\end{array}\right)
=\left(\begin{array}{cc}0 & -D_L\Theta\\ 0 & 0\end{array}\right).
\]
We hence obtain 
\[ \hh \l \Xi,\Xi'\r=
\hh \l \check{\Xi},\, \partial_A \check{\Xi}\r
=\on{tr}(\check{\Xi}\partial_A\check{\Xi})=\Theta\,(D_L\Theta). 
\]
as required. Consider next the boundary terms. Again, we find it more convenient to express these in terms of $\check{\Xi}$:
\begin{align*} \hh\, \l\Xi_{{x}_0},\, q^*\theta^L\r&=-\hh \l \check{\Xi}_0,\, \check{\Xi}_1\r\\
&=-\on{tr}\ \left(\begin{array}{cc}\hh \Theta_0' & -\Theta_0 \\  \hh \Theta_0''+\T_0\Theta_0& -\hh\Theta_0' \end{array}\right)
\left(\begin{array}{cc}\hh \Theta_1' & -\Theta_1 \\  \hh \Theta_1''+\T_1\Theta_1& -\hh\Theta_1' \end{array}\right)
\\
&=2\T_0\Theta_0\Theta_1+\hh(\Theta_0''\Theta_1-\Theta_0'\Theta_1'+\Theta_0\Theta_1'').
\end{align*}
Here we used $\T_1=\T_0$ by periodicity of the Hill potential. 
\end{proof}
\begin{remark}
	The proof showed in particular that the  linear connection $\nabla=\delta+A$ corresponding to the Hill operator $L=\f{d^2}{d x^2}+\T(x)$ is given by \eqref{eq:Aconnection}. 
\end{remark}

The inclusion $\wh{\iota}$ is  $\Diff_+(\CC)$-equivariant. Indeed, the action of $\Diff_+(\CC)$ on jets determines an action on $E$ by linear bundle automorphisms, preserving orientation, and hence an action on $Q$ by principal bundle  automorphisms. This defines a 
group homomorphism 
\[ \Diff_+(\CC)\hra  \Aut_+(Q)\] which splits the exact sequence \eqref{eq:gaugesequence}. 
It lifts to  a group homomorphism 
\begin{equation}\label{eq:grouphom} 
\Diff_\Z(\wt\CC)\hra  \Aut_\Z(\pi^*Q)
\end{equation}
splitting \eqref{eq:gaugesequence1}. From the coordinate-free construction, it is clear that the embedding \eqref{eq:iota} is  equivariant with respect to \eqref{eq:grouphom}. Let us also give the coordinate expression. 
	\begin{proposition}
	For $\CC=S^1=\R/\Z$, 	and $Q=S^1\times G$, 
	the canonical inclusion
	\[\Diff_\Z(\wt\CC)\hra  \Aut_\Z(\pi^*Q)=C^\infty(S^1,G)\rtimes \Diff_\Z(S^1)\]
	is given by 
	$\FF^{-1}\mapsto ([g],\FF^{-1})$ where 
	\[ g= 
		\left( \begin{array}{cc}\FF'(x)^{\f{1}{2}}&-\hh \FF''(x)\FF'(x)^{-\f{3}{2}}  \\ 0& \FF'(x)^{-\f{1}{2}}
	\end{array}\right).
	\]
	The description of $ \Diff_+(\CC)\hra  \Aut_+(Q)$ is analogous.
\end{proposition}
\begin{proof}
	The pullback of $u=f(x) |\partial x|^{-\f{1}{2}}$ is 
	\[ \FF^* u=(\FF^{-1})\cdot u=\FF'(x)^{-\f{1}{2}}f(\FF(x))\ |\partial x|^{-\f{1}{2}}.\] 
	Replacing $x$ with $x+t$, and taking the first order 
	Taylor approximation, the coefficient function is replaced by 
	\[ \FF'(x)^{-\f{1}{2}}f(\FF(x))+t\,\Big(\FF'(x)^{\f{1}{2}}f'(\FF(x))-\hh \FF''(x)\FF'(x)^{-\f{3}{2}}f(\FF(x))\Big) +O(t^2)\]
	so that 
		\[\left(\begin{array}{c} f'(x) \\f(x) \end{array}\right)\mapsto 
	\left( \begin{array}{cc}\FF'(x)^{\f{1}{2}}&-\hh \FF''(x)\FF'(x)^{-\f{3}{2}}  \\ 0& \FF'(x)^{-\f{1}{2}}
	\end{array}\right)
	\left(\begin{array}{c} f'(\FF(x)) \\f(\FF(x)) \end{array}\right)
	\]
\end{proof}

\begin{proposition}\label{prop:groupoidrestriction}
	The symplectic groupoid $\G_1\rra \vir^*_1(\CC)$ is the restriction (as a groupoid) 
	of 	$\Gau_0(Q)\ltimes \A(Q)\rra \A(Q)$ to the Drinfeld-Sokolov slice $\vir^*_1(\CC)\subset \A(Q)$.  
\end{proposition} 

\begin{proof}
	By the general property \eqref{eq:recovery} of Morita equivalences, the groupoid 
	$\Gau_0(Q)\ltimes \A(Q)$ is recovered from the equivalence bimodule $\P(Q)$ as 
	\[ \Gau_0(Q)\ltimes \A(Q)=(\P(Q)\times_{\wt{G}}\P(Q))/G,\]
	with the symplectic form induced by $\pr_1^*\varpi_\P-\pr_2^*\varpi_\P$.  
	The restriction of this groupoid to $\vir^*_1(\CC)$ is given by arrows for which both source and target are in 
	$\vir^*_1(\CC)$. But this is exactly $\G_1=(\D(\CC)\times_{\wt{G}_+}\D(\CC))/G$, and the pullback of the symplectic form 
	is induced by $\pr_1^*\varpi_\D-\pr_2^*\varpi_\D$. 
\end{proof}

To summarize, the Morita equivalence \eqref{eq:corr0c} for the coadjoint Virasoro action is obtained from the Morita equivalence 
\eqref{eq:corrLG4} by restriction. This extends to the Hamiltonian spaces: If $(\M,\omega_\M)$ is a Hamiltonian $\on{Gau}_0(Q)$-space,
 then the pre-image of $\vir^*_1(\CC)\subset \A(Q)$ under the momentum map $\Phi_\M\colon \M\to \A(Q)$ is a Hamiltonian Virasoro space.

\subsection{Drinfeld-Sokolov reduction}\label{subsec:ds}

The result $\varpi_{\D}=\wh{\iota}^*\varpi_{\P}$ has a conceptual explanation. We shall follow 
the coordinate-free description of Drinfeld-Sokolov reduction 
\cite{dri:kdv} given in  Segal's paper \cite{seg:geo}; see also \cite{fre:ver}. 
Let 
\[ \PP\to \CC=\wt\CC/\Z\] 
be an oriented rank $1$ projective bundle, with a distinguished section $\sigma\in \Gamma(\PP)$. 
 Let $Q\to \CC$ be the associated principal $G=\PSL(2,\R)$-bundle. The gauge group 
$\G=\Gau(Q)$ acts by projective transformations on $\PP$. The section $\sigma$ determines subgroups 
\begin{equation}\label{eq:groups} \ca{N}\subset \ca{B}\subset \ca{G}=\Gau(Q),\end{equation}
where $\ca{B}$ are the gauge transformations fixing $\sigma$, and $\ca{N}$ are those which furthermore act trivially on 
$\sigma^*V\PP$, where $V\PP\subset T\PP$ is the vertical bundle. The corresponding Lie algebras 
\begin{equation}\label{eq:algebras} \on{Lie}(\ca{N})\subset \on{Lie}(\ca{B})\subset \on{Lie}(\ca{G})=\gau(Q)\end{equation}
act as infinitesimal gauge transformations on $\PP$, and so are realized as spaces of vertical vector fields on $\PP$.  The metric on $\g=\mf{sl}(2,\R)$  induces a bundle metric on $Q\times_G \g$, and defines a non-degenerate $C^\infty(\CC)$-valued bilinear form on $\mf{gau}(Q)$. Using local trivializations, we see that
$\on{Lie}(\ca{G})/\on{Lie}(\ca{B})=\Gamma(\sigma^*V\PP)$ is non-singularly paired with $\on{Lie}(\ca{N})$. Combined with integration over $\CC$, this identifies 
\begin{equation}\label{eq:ds-identify} \Omega^1(\CC,\sigma^* V\PP)=\on{Lie}(\ca{N})^*\end{equation}
(the smooth dual). The space  $\A(Q)$ of principal connections may be regarded as the space 
of projective connections on $\PP$. Given such a connection,  the composition of $T\sigma\colon T\CC\to T\PP$ 
 followed by vertical projection defines a bundle map $T\CC\to \sigma^*V\PP$, which we may also think of as an element of 
 \eqref{eq:ds-identify}. This gives a $\ca{B}$-equivariant map
\begin{equation}\label{eq:psimap} \Psi\colon \A(Q)\to  \on{Lie}(\ca{N})^*.
\end{equation}
By construction, $\Psi(A)=0$ if and only if $\sigma$ is an $A$-horizontal section. 
\begin{proposition}\label{prop:Nmomentum}
	The restriction of the central extension $\wh{\mf{gau}}(Q)$ to $\on{Lie}(\ca{N})\subset \gau(Q)$ is canonically trivial. 
\end{proposition}
\begin{proof}	
	By definition,  $\wh{\gau}(Q)$ consists of affine-linear functionals $\wh\xi\colon \A(Q)\to \R$ 
	whose underlying linear functional $\Omega^1(\CC,Q\times_G\g)\to \R$ is given by an element 
	$\xi\in \gau(Q)$
	 (via pairing and integration). For $\xi\in  \on{Lie}(\ca{N})$, the map  
	 \[ \wh{\xi}\colon \A(Q)\to \R,\ \ A\mapsto \l\Psi(A),\xi\r\] 
	 is a distinguished lift. This defines a splitting. 
\end{proof}

The inclusion $\A(Q)=\wh{\gau}(Q)^*_1\hra \wh{\gau}(Q)^*$  is the momentum map for the gauge 
action. Hence,  we may regard \eqref{eq:psimap}  as a momentum map (equivariant in the usual sense) for the action of the subgroup $\ca{N}$.  \medskip

Let us now return to the case that $\PP$ is the projectivization of $E=J^1(|\Lambda|_\CC^{-\f{1}{2}})$. The kernel of the natural map to 
$|\Lambda|_\CC^{-\f{1}{2}}$ defines a rank $1$ subbundle $\ell\subset E$, 
which in turn determines a section  $\sigma$ of $\PP$. 
From the jet sequence , we have 
\begin{equation}\label{eq:ell} \ell=T^*\CC\otimes \ |\Lambda|_\CC^{-\f{1}{2}}\cong |\Lambda|_\CC^{\f{1}{2}},\end{equation}
This gives a canonical identification 
\[ \sigma^* V\PP=E/\ell\otimes \ell^*\cong  |\Lambda|^{-\f{1}{2}}_\CC\otimes  |\Lambda|^{-\f{1}{2}}_\CC
= |\Lambda|^{-1}_\CC\cong T\CC. \]
Thus $\on{Lie}(\ca{N})^*\cong \Omega^1(\CC,T\CC)
=C^\infty(\CC)$ are just ordinary functions, and $\Psi$ becomes a $\ca{B}$-equivariant  map 
\[ \Psi\colon \A(Q)\to\on{Lie}(\ca{N})^*= C^\infty(\CC).\]


\begin{theorem}[Drinfeld-Sokolov] \label{th:ds} 
	For any nonvanishing function $f\in C^\infty(\CC)$, the $\ca{N}$-action on the level set $\Psi^{-1}(f)$ is free. In particular, this applies to the constant function $f=-1$.   The inclusion 
	\begin{equation}\label{eq:dsmap} \on{Hill}(\CC)=\vir^*_1(\CC)\to \A(Q)\end{equation}
	takes values in $\Psi^{-1}(-1)$, and is a global slice for the action on this level set. This identifies 
\[ \vir^*_1(\CC)\cong \Psi^{-1}(-1)/\ca{N}.\]
\end{theorem}	
\begin{proof} (See, e.g., \cite{khe:inf}.)
We may assume $\CC=S^1=\R/\Z$.  Following the discussion from the proof of Theorem
\ref{th:pullback}, and using the trivialization $E\cong \CC\times \R^2$ constructed there, the quotient map $E\to E/\ell$ is projection to the second component of $\R^2$. 
Hence the subbundle $\ell$ is the subbundle of $\CC\times (\R\oplus 0)\subset \CC\times \R^2$, and $\sigma\in \Gamma(\PP)$   
is the constant section  $(1:0)$ of $S^1\times \RP(1)$. 

Let $N\subset B\subset G$ be the connected subgroups whose Lie algebras are the 
subalgebras of $\g=\mf{sl}(2,\R)$, consisting of strictly upper triangular and lower triangular matrices, respectively. 
Then $\ca{N}=C^\infty(S^1,N)\subset \ca{B}=C^\infty(S^1,B)$.  Using the metric on $\g$ to identify $\g\cong \g^*$, 
the space $\on{Lie}(\ca{N})^*$ is identified with $C^\infty(S^1,\mf{n}^-)$, where $\mf{n}^-$ are the strictly lower triangular 
matrices. The map $\Psi$ takes a connection 1-form $A\in \Omega^1(S^1,\g)\cong C^\infty(S^1,\g)$ to its lower left corner, $A_{21}$. The gauge action of $\ca{N}$ reads as 
\[ \left(\begin{array}{cc} 1& \chi\\ 0 & 1
\end{array}\right)\cdot  \left(\begin{array}{cc} A_{11}& A_{12}\\ A_{21} & A_{22}
\end{array}\right)=\left(\begin{array}{cc} A_{11}+\chi A_{21}& A_{12} -2\chi A_{11}-\chi^2 A_{21}-\chi'\\A_{21} & A_{22}-\chi A_{21}
\end{array}\right).\] 
We see  that on the subset where $A_{21}=f$ has no zeroes, the action is free. Furthermore, for any fixed $f$, the subset where the 
diagonal entries are zero serves as a global slice.  In particular, this is true for $f=-1$, proving that matrices of the form 
\eqref{eq:Aconnection}
are a global slice for the $\ca{N}$-action on $\Psi^{-1}(-1)$. 
\end{proof}
Let 
\[ \wh{\Psi}=\Psi\circ p \colon \P(Q)\to C^\infty(\CC)\] 
be the lifted map. Taking pre-images, Theorem \ref{th:ds} implies that the embedding  $\wh{\iota}\colon \DC\hra \P(Q)$ is a global slice for the $\ca{N}$-action on 
\[\wh{\Psi}^{-1}(-1)\subset \P(Q).\]
For $\xi\in \on{Lie}(\ca{N})$, we have that $\iota(\xi_{\P(Q)})\varpi_\P=-\l \d \wh{\Psi},\xi\r$ as a consequence of the momentum map 
property of $\varpi_P$ (Theorem \ref{th:varpiforloopgroups}). The pullback of this expression to $\wh{\Psi}^{-1}(-1)$ vanishes; hence $i_{\wh{\Psi}^{-1}(-1)}^*\varpi_\P$ is   
$\ca{N}$-basic. Together with $\wh{\iota}^*\varpi_\P=\varpi_\D$, this shows that $\varpi_\D$ is the reduction of $\varpi_\P$. 

Similarly, the symplectic groupoid $\ca{G}_1\rra \vir^*_1(\CC)$ may be regarded as the symplectic reduction of 
$\Gau_0(Q)\ltimes \A(Q)\rra \A(Q)$ under the action of $\ca{N}\times \ca{N}$.  This follows by the argument from the proof of Proposition \ref{prop:groupoidrestriction}, since the two groupoids are described as 
$(\P(Q)\times_{\wt{G}}\P(Q))/G$ and $(\D(\CC)\times_{\wt{G}_+}\D(\CC))/G$, respectively. 

\begin{appendix}

\section{Morita equivalence of (quasi) symplectic groupoids}\label{app:a}

\subsection{Morita equivalence of Lie groupoids}
Let $\G\rra M$ be a Lie groupoid, with $M\subset \G$ the subset of units, and with 
source and target maps denoted $\sz,\tz\colon \G\to M$. The groupoid product will be written as $g_1g_2$ for $\sz(g_1)=\tz(g_2)$. 
A left action of  $\G\rra M$ on a manifold $P$, depicted as 
\[ 
\xymatrix{ \G\ar[d]<2pt>\ar@<-2pt>[d] & P\ar[ld]^{J}\\
	M & & 	}
\] 
involves a smooth \emph{anchor map} $J\colon P\to M$ and a smooth \emph{action map} 
\[ \G\times_M P\to P,\ \ (g,p)\mapsto g\cdot p\]
(using the fiber product with respect to $\sz$ and $J$) 
satisfying $g_1\cdot (g_2\cdot p)=(g_1g_2)\cdot p$ and $m\cdot p=p$ for units $m\in M$. For instance, $\G$ acts on itself by 
left multiplication (here $J=\tz$), and on its units $M$ by $g\cdot m=\tz(g)$ (here $J=\on{id}_M$). 
Similarly,  right actions of $\G$ along $J\colon P\to M$
\[ 
\xymatrix{  P\ar[rd]_{J}& \G\ar[d]<2pt>\ar@<-2pt>[d]\\
	& M	}
\] 
are described by maps
$P\times_M \G\to P,\ (p,g)\mapsto p\cdot g$ (using the fiber product with respect to $J$ and $\tz$)
satisfying $(p\cdot g_1)\cdot g_2=p\cdot (g_1g_2),\ p\cdot m=p$.  A right action of $\G$ may be regarded as a left action, using the groupoid inversion to define $g\cdot p$ as $p\cdot g^{-1}$. 

\begin{example}
	Let $\G=G\ltimes M\rra M$ be the action groupoid of a (left) $G$-action on $M$. 
	Thus $\tz(g,m)=g\cdot m,\ \sz(g,m)=m$ and $(g_1,m_1)(g_2,m_2)=(g_1g_2,m_2)$ (defined when $m_1=g_2\cdot m_2$). 
	A left $\G$-action on $P$ along $J\colon P\to M$ is the same as a $G$-action on $P$ for which the map $J$ is $G$-equivariant. 
\end{example}

A left $\G$-action is called a \emph{principal left action} \cite[Chapter 5.7]{moe:fol} if the orbit space 
$B=\G\backslash P$ is a manifold, with quotient map a surjective submersion, and the map 
\[ \G\times_M P\to P\times_B P,\ 
(g,p)\mapsto (g\cdot p,p)\] is a diffeomorphism. Similarly, one defines principal right actions. 
A \emph{Morita equivalence} between Lie groupoids  $\G_1\rra M_1$ and $\G_2\rra M_2$ is described by a diagram
\begin{equation}\label{eq:moritaequivalence}
 \begin{tikzcd}
[column sep={5.5em,between origins},
row sep={4em,between origins},]
\G_1    \arrow[d,shift left]\arrow[d,shift right]   & Q \arrow [dr, "J_2"'] \arrow[dl, "J_1"]  & \G_2    \arrow[d,shift left]\arrow[d,shift right] \\
M_1 & & M_2
\end{tikzcd}
\end{equation}
where $Q$ is a manifold (the \emph{Hilsum-Skandalis bibundle}), equipped with a principal left action of $\G_1$ along $J_1$ and a principal right action of $\G_2$ along $J_2$, such that the two actions commute, $J_1$ is 
the quotient map to $M_1=Q/\G_2$, and 
 $J_2$ is the quotient map 
to $M_2=\G_1\backslash Q$.  We shall denote the $(\G_1,\G_2)$-bi-action by $\A\colon \G_1\times_{M_1}Q\times_{M_2}\G_2\to Q$. Some features:

\begin{enumerate}
\item Morita equivalence is an equivalence relation. It is \emph{reflexive}: $Q=\G$ gives a Morita equivalence from $\G$ to itself. It is \emph{symmetric}:
If $Q$ gives a Morita equivalence from $\G_1$ to $\G_2$, then the opposite bimodule $Q^\op$ (equal to $Q$ as a manifold, with 
the roles of $J_1,J_2$ interchanged, and the new bi-action $\A^\op(g_2,p,g_1)=\A(g_1^{-1},p,g_2^{-1})$) gives a  Morita equivalence from $\G_2$ to $\G_1$. It is \emph{ transitive}: given 
Morita equivalences $Q$ from $\G_1$ to $\G_2$ and $Q'$ from $\G_2$ to $\G_3$, then 
\[ Q\diamond Q'=\frac{Q\times_{M_2} Q'}{\G_2}\]
(quotient under equivalence relation $(qg_2,q')\sim (q,g_2q')$)
is a Morita equivalence from $\G_1$ to $\G_3$. 
\item A Morita equivalence identifies the `transverse geometry' of the groupoids. In particular, it
induces a homeomorphism of orbit spaces $\G_1\backslash M_1\cong \G_1\backslash Q/\G_2\cong M_2/\G_2$, isomorphisms of 
isotropy groups 
\[ (\G_1)_{x_1}\cong (\G_2)_{x_2},\] 
for $x_i=J_i(q),\ i=1,2$, and isomorphisms of their representations on the normal spaces to orbits, 
$T_{x_i}M_i/T_{x_i}(\G_i\cdot x_i)$. 
\item\label{it:b}
The groupoids $\G_1,\G_2$ may be recovered from $Q$ via 
\[ \G_1=Q\diamond Q^\op,\ \ \G_2=Q^\op\diamond Q.\]
%

\end{enumerate}

\begin{example}\label{ex:relevant}
	Let $K$ be a Lie group, and $\pi\colon P\to M$ a left principal $K$-bundle. 
	Suppose $q\colon P\to N$ is an equivariant submersion onto another $K$-manifold $N$. Let $P\times_N P\rra P$ be the submersion groupoid (a subgroupoid of the pair groupoid), and $(P\times_N P)/K\rra M$ its quotient under the diagonal action of $K$. 
	We obtain an Morita equivalence with the 
	action groupoid $K\ltimes N\rra N$:
	\[  \begin{tikzcd}
	[column sep={5.5em,between origins},
	row sep={4em,between origins},]
	(P\times_N P)/K \arrow[d,shift left]\arrow[d,shift right]
	& P \arrow[dl, "\pi"] \arrow [dr, "q"'] & K\ltimes N \arrow[d,shift left]\arrow[d,shift right]
	\\ M& & N \end{tikzcd}\]
	
	Taking $N=P$, we obtain a Morita equivalence between the action groupoid $K\ltimes P\rra P$ and the trivial groupoid $M\rra M$; 
	for $N=\pt$ we obtain a Morita equivalence between the gauge groupoid $\G(P)=(P\times P)/K\rra M$ and the group $K\rra \pt$.
	
\end{example}

\begin{example}[Morita equivalence of group actions]
	Let $Q$ be a manifold with commuting actions of Lie groups $K_1$ and $K_2$, both of which are principal actions. Let $M_1=Q/K_2$ with the induced action of $K_1$ and $M_2=Q/K_1$ with the induced action of $K_2$. Then $Q$ defines a Morita equivalence of the action groupoids $K_1\ltimes M_1\rra M_1$ and $K_2\ltimes M_2\rra M_2$.  
\end{example}

\begin{example}
	The Morita equivalence $Q$ gives rise to a 1-1 correspondence between $\G_1$-spaces and $\G_2$-spaces, as follows. 
	Given a $\G_2$-action on a manifold $P_2$, along a map $\Phi_2\colon P_2\to M_2$,  one obtains a $\G_1$-action on the manifold 
	\[ P_1=Q\diamond P_2=(Q\times_{M_2}P_2)/\G_2\]
	along the induced map $\Phi_1\colon P_1\to M_1$. 
	One recovers  $P_1=Q^-\diamond P_2$. In fact, one may understand this correspondence as a Morita equivalence of the two action groupoids 
		\[  \begin{tikzcd}
	[column sep={5.5em,between origins},
	row sep={4em,between origins},]
	\G_1\ltimes P_1 \arrow[d,shift left]\arrow[d,shift right]
	& R \arrow[dl, "f_1"] \arrow [dr, "f_2"'] &  \G_2\ltimes P_2  \arrow[d,shift left]\arrow[d,shift right]
	\\ P_1& & P_2 \end{tikzcd}\]
	with the equivalence bimodule $R=Q\times_{M_2}P_2\cong Q^\op\times_{M_1} P_1$. There is a natural map $\Psi\colon R\to Q$ inducing the maps $\Phi_i$.
	As consequences, if $p_i\in P_i$ correspond to each other in the sense that $p_i=f_i(r)$, then the normal spaces to the 
	orbits $T_{p_i}P_i/T_{p_i}(\G_i\cdot p_i)$, the stabilizer groups 
	$(\G_i)_{p_i}$, as well as their representations on the normal spaces, are identified.  Furthermore, one has  canonical identifications 
	\begin{equation}\label{eq:kernelidentification} \ker(T_{p_1}\Phi_1)\cong \ker(T_{p_2}\Phi_2),\end{equation}
	since these are identified with $\ker(T_r\Psi)$, and similarly for the cokernel.  
\end{example}

For further details and examples, we refer to \cite{hoy:lie}.

\subsection{Quasi-symplectic groupoids}
A 2-form $\omega\in \Omega^2(\G)$ on a Lie groupoid $\G\rra M$ is called  \emph{multiplicative} if 
\begin{equation}\label{eq:multiplicative} \Mult_\G^*\omega=\pr_1^*\omega+\pr_2^*\omega\end{equation}
where $\Mult_\G\colon \G^{(2)}\to \G$ is the groupoid multiplication and $\pr_1,\pr_2\colon \G^{(2)}\to \G$ are the two projections (here $\G^{(2)}$ denotes the space of composable arrows). This property implies that $\omega$ 
pulls back to $0$ on $M$ (see \cite{bur:int,xu:mom}). 
It therefore defines a bundle map 
\begin{equation}\label{eq:rhomap} \rho\colon \on{Lie}(\G)=\nu(\G,M)\to T^*M\end{equation}
such that $\l \rho([w]),\,v\r=\omega(w,v)$ for $v\in TM$ and $[w]\in \nu(\G,M)$
(represented by an element  $w\in T\G|_M$). Letting $\sigma^L,\sigma^R$ denote the left-,right-invariant vector fields defined by 
$\sigma\in \Gamma(\on{Lie}(G))$,  
\begin{equation}\label{eq:contraction}
 \iota(\sigma^L)\omega=\sz^*\rho(\sigma),\ \ \iota(\sigma^R)\omega=\tz^*\rho(\sigma).\end{equation}
%

The pair $(\G,\omega)$ is called a \emph{symplectic groupoid} \cite{cos:gro} if $\omega$ is a multiplicative symplectic 2-form. In this case, the manifold of units acquires a Poisson structure, and the symplectic groupoid is called its integration.  

\begin{example}\label{ex:cotangentgroupoid}
	The main example of a symplectic groupoid is the cotangent groupoid $T^*G\rra \g^*$ of a Lie group. Using left trivialization $T^*G=G\times \g^*$, the groupoid structure is that as an action groupoid for the coadjoint action; in particular, $\tz(k,\mu)=k\cdot \mu,\ \sz(k,\mu)=\mu$.  
	The 2-form is given by the canonical symplectic form of the cotangent bundle:
	%
	\begin{equation}\label{eq:cotangentbundle} \omega=-\d\l\mu,\theta^L\r
	\end{equation}
	with the left-invariant Maurer-Cartan form $\theta^L\in \Omega^1(G,\g)$. 
	The induced Poisson structure on $\g^*$ is the standard Kirillov Poisson structure. 
\end{example}

\emph{Quasi-symplectic groupoids} \cite{bur:int,lau:qua,xu:mom} (called \emph{twisted pre-symplectic groupoids} in \cite{bur:int}) are a generalization of symplectic groupoids. Here $\omega$ need not be closed, but instead should satisfy 
 \begin{equation}\label{eq:quasiclosed} \d\omega=\sz^*\eta-\tz^*\eta\end{equation}
 for a given closed 3-form $\eta\in \Omega^3(M)$. As a non-degeneracy condition, one requires that the map  
  \[ (\a,\rho)\colon \on{Lie}(\G)\to TM\oplus T^*M\]
given by the anchor $\a$ the map  \eqref{eq:rhomap}, is an inclusion as a Lagrangian (i.e., maximal isotropic) subbundle. 
 In this case, the range of $(\a,\rho)$ defines an $\eta$-twisted Dirac structure on $M$, with $(\G,\omega)$ its integration in the sense 
 of \cite{bur:int}.
 	
 \begin{remark}
 By \cite[Corollary 3.4]{bur:int}, a multiplicative 2-form $\omega$ on a source-connected 
  Lie groupoid $\G\rra M$, satisfying  \eqref{eq:quasiclosed}  for a given closed 3-form $\eta$, is  uniquely determined by its values along the units.  (This is not true for source-disconnected groupoids, in general.) 
  Since $\omega$ pulls back to zero on $M$, it follows that $\omega$ is uniquely determined by the associated map $\rho$.
 \end{remark}
 
 \begin{remark}
 For any quasi-symplectic groupoid $(\G,\omega)$, one may define a family of  quasi-symplectic groupoids $(\G,t\omega)$
 for all $t\neq 0$, with the corresponding 3-form $t\eta$. 
 The corresponding Dirac structures are related by the automorphism $(v,\mu)\mapsto (v,t\mu)$ 
 of $TM\oplus T^*M$. 
\end{remark} 

\subsection{Examples}
Many  examples of quasi-symplectic groupoids arise as action groupoids $\G=G\ltimes M\rra M$ for some $G$-action on $M$. 
Suppose $\eta\in \Omega^3(M)$ is a $G$-invariant 3-form, and $\alpha\colon \g\to \Omega^1(M)$ a $G$-equivariant map, 
so that $X\mapsto \eta-\alpha(X)$ is equivariantly closed (in the sense of Cartan's equivariant de Rham complex). That is, 
\[ \d\eta=0,\ \ \ \iota(X_M)\eta=-\d\alpha(X),\ \ \ \iota(X_M)\alpha(X)=0.\] 
Put $\chi(X,Y)=\iota(X_M)\alpha(Y)=-\chi(Y,X)$. 
\begin{proposition}\cite[Section 6.4]{bur:int} 
	\label{prop:actiongroupoidtype}
The 2-form 	\[ \omega=-\l\alpha,\theta^L\r+\hh \chi(\theta^L,\theta^L)\in \Omega^2(G\times M)\]
is a multiplicative 2-form satisfying $\d\omega=\sz^*\eta-\tz^*\eta$. 
The map $\varrho$ (cf.~ \eqref{eq:rhomap}) coincides with $\alpha$ under the identification $\on{Lie}(\G)=M\times \g$.
Hence, $\omega$ define a quasi-symplectic groupoid if and only if the range of 
the map $M\times \g\to TM\oplus T^*M,\ (m,X)\mapsto (X_M|_m,\alpha(X)|_m)$ is Lagrangian.
\end{proposition}\

The groupoid itself is acted upon by $G\times G$, where 
$ (g_1,g_2)\cdot (g,m)=(g_1gg_2^{-1},g_2\cdot m)$. 
The 2-form $\omega$ is invariant under this action, and its contractions with the generating vector fields are 
\[ \iota\big((X_1,X_2)_{G\ltimes M}\big)\omega=-\tz^*\alpha(X_1)+\sz^*\alpha(X_2).\]
as a special case of \eqref{eq:contraction}.

\begin{example}
	For $M=\g^*$ with the coadjoint action, the choice 
	\[\eta=0,\ \ \ \alpha(X)=\l\d \mu,X\r\] reproduces the example \ref{ex:cotangentgroupoid} of a symplectic groupoid. In this case, 
	\[ \chi(X_1,X_2)=\l\mu,[X_1,X_2]\r.\] 
\end{example}

\begin{example}\label{ex:gvalued}
Let $G$ be a Lie group with an invariant nondegenerate symmetric bilinear form 	$\l\cdot,\cdot\r$ on its Lie algebra $\g$. Then 
\[\eta=\f{1}{12}\l\theta^L,\,[\theta^L,\theta^L]\r,\ \  \alpha(X)=\hh\,\l\theta^L+\theta^R, X\r\]
defines a closed equivariant 3-form, with respect to the conjugation action of $G$ on itself. 
In this case, 
\[ \chi(X_1,X_2)=\hh \l (\Ad_g-\Ad_{g^{-1}})X_1,\, X_2\r.\] 
The resulting 2-form on the action groupoid $\G=G\ltimes G\rra G$ was introduced in \cite{al:mom} as the q-Hamiltonian 2-form on the `double'. 
The interpretation of the double as quasi-symplectic groupoid is from \cite{bur:int,xu:mom}. 
\end{example}

\begin{example}\label{ex:twistedgvalued}
Given an automorphism $\epsilon$ of $G$, one may consider the  twisted conjugation action of $G$ on itself, $a\cdot g=a\,g\,\epsilon(a)^{-1}$. 
Suppose the induced automorphism of $\g$ (still denoted by $\epsilon$) preserves the metric. 
Taking $\eta$ to be the Cartan 3-form as in Example \ref{ex:gvalued}, but using $\alpha(X)=\hh\l \theta^L,\epsilon(X)\r+\hh\l\theta^R,\,X\r$, one obtains a quasi-symplectic groupoid. This appears in the theory of twisted group-valued moment maps \cite{boa:twi,me:conv}.
\end{example}

\begin{example}\label{ex:centralext}
	Suppose we are given a central extension 
	\[ 0\to \R\to \wh{\g}\to \g\to 1,\]
	with a $G$-action on $\wh{\g}$ lifting the adjoint action. 
	Dually, we have
	\[ 0\to \g^*\to \wh{\g}^*\to \R\to 0.\]
	All level sets $\wh{\g}^*_\mathsf{c}$ (pre-image of $\mathsf{c}\in\R$) are $G$-invariant affine-linear Poisson submanifolds. 
	The action groupoid  $G\times \wh{\g}^*_\mathsf{c}\rra  \wh{\g}^*_\mathsf{c}$ for the level $\mathsf{c}$ action becomes a symplectic groupoid, for the choice 
	\[ \eta=0,\ \alpha(X)=\l\d\mu,X\r\] 
	(thus, $\alpha$ is the `tautological 1-form on the affine space $\g^*_{\mathsf{c}}$). 
	The resulting $\chi$ depend on the level $\mathsf{c}$ (since it depends on the $G$-action). 
	To write explicit formulas, it is convenient to choose a splitting of the central extension. Thus suppose 
	 $\wh{\g}=\g\oplus \R$
	with the bracket  $ [(X_1,t_1),(X_2,t_2)]=([X_1,X_2],\sigma(X_1,X_2))$,
	for a Lie algebra cocycle 
	\[ \sigma\colon\g\times \g\to \R.\] 
	The $G$-action on 
	$\wh{\g}$ reads as $g\cdot(X,t)=(\Ad_gX,\ t+\lambda_g(X))$
	for a 1-cocycle on $G$ with with values in the coadjoint representation, 
	\[\lambda\colon G\to \g^*,\ g\mapsto \lambda_g.\] 
	(The cocycle property reads $\lambda_{g_1g_2}=\lambda_{g_2}+g_2^{-1}\cdot \lambda_{g_1}$.) The splitting also identifies
	$\wh{\g}^*_{\mathsf{c}}\cong \g^*$, with the level $\mathsf{c}$ coadjoint action 
	\[g\cdot\mu=\Ad_g(\mu-\mathsf{c}\lambda_g).\]
	A calculation shows 
	$\chi(X_1,X_2)=	
	\l \mu,[X_1,X_2]\r+\mathsf{c}\,\sigma(X_1,X_2)$, 
	so that 
	\begin{equation}\label{eq:omegaformula}
	 \omega=-\l\d\mu,\theta^L\r+\hh \l \mu,[\theta^L,\theta^L]\r+\f{\mathsf{c}}{2}\,\sigma(\theta^L,\theta^L).\end{equation}
\end{example}	

\begin{remark}
	Given a Lie group extension $\wh{G}\to G$ integrating the Lie algebra extension, this symplectic groupoid may also be obtained  
	by symplectic reduction from 
	the cotangent bundle $T^*\wh{G}\rra \wh{\g}^*$. 
\end{remark}

We conclude with a simple example of a symplectic groupoid that is not an action groupoid. 

\begin{example}\label{ex:baby2}
	Recall that $T^*\R=\R^2$, with the  standard symplectic form $\omega=\d x\wedge \d y$, and groupoid structure  
	given by fiberwise addition, is the source 1-connected symplectic groupoid integrating $\R$ with the zero Poisson structure. The other source connected  symplectic groupoids integrating $\R$ with the zero Poisson structure
	are obtained quotients of $T^*\R$ by closed subgroupoids $\ca{K}\rra \R$ with discrete fibers. Thus, each $\ca{K}_x\subset \R$ is either trivial, or is a cyclic subgroup. In order for the quotient to be a manifold, it is necessary and sufficient that there is some open neighborhood of the zero section in $T^*\R$ does not contain non-trivial elements of $\ca{K}$.  The
	resulting $\G\rra \R$ is a family of Lie groups $\G_x=\R/\ca{K}_x$. 
	The symplectic form $\d x\wedge \d y$ 
	descends to $\G$, making the latter into a $\sz$-connected symplectic groupoid. It is an action groupoid if and only if the family of subgroups  $\ca{K}_x$ is constant. 
\end{example}

\subsection{Hamiltonian spaces}
A (left) \emph{Hamiltonian space} $(P,\omega_P)$ for a quasi-symplectic groupoid $\G\rra M$, with 2-form $\omega\in \Omega^2(\G)$ and 3-form $\eta\in \Omega^3(M)$, is given by a manifold $P$ with a 2-form $\omega_P$, equipped with a (left)  $\G$-action along a map $\Phi\colon P\to M$, such that 
\begin{enumerate}
   \item\label{it:b1} $\A^*\omega_P=\pr_\G^*\omega+\pr_P^*\omega_P$	
	\item\label{it:b2} $\d\omega_P=-\Phi^*\eta$
	\item\label{it:b3} $\ker(\omega_P)\cap \ker(T\Phi)=0$. 
\end{enumerate}
Here $\pr_\G,\pr_P$ are the 
projections from $\G\times_M P$ to $\G$ and $P$, respectively, and $\A \colon \G\times_M P\to P$ is the action map. 
We remark that \eqref{it:b1} encodes both an invariance property of $\omega_P$ and  a moment map property for $\Phi$ (which enters 
the definition of the fiber product $\G\times_MP$). 
For the case of symplectic groupoids, conditions \eqref{it:b2} and \eqref{it:b3} are equivalent to $\omega_P$ being symplectic, and one recovers the notion   
of Hamiltonian actions of symplectic groupoids due to Mikami-Weinstein \cite{mik:mom}.  The generalization to quasi-symplectic groupoids is due to Xu \cite{xu:mom}. 

\begin{remark}
The conditions imply that $\mu\in M$ is a regular value of $\Phi$ if and only if the isotropy group $\G_\mu$ acts locally freely on
the level set $\Phi^{-1}(\mu)$. The symplectic form descends to the orbifold $P_\mu=\Phi^{-1}(\mu)/\G_\mu$ (`symplectic quotient').  
See \cite[Theorem 3.18]{xu:mom}
\end{remark}

\begin{remark}
	One may similarly define a right Hamiltonian space $(P,\omega_P)$, as a manifold $P$ with a right-action of $G\rra M$ along a map 
	$\Phi\colon P\to M$ satisfying the same conditions as above, but with (b) replaced by $\d\omega_P=\Phi^*\eta$.  Using inversion, a right Hamiltonian space
	for $\G$ may be regarded as a left Hamiltonian space for the 
 quasi-symplectic groupoid $(\G,-\omega)$.
\end{remark}

\begin{example}
	$(P,\omega_P)=(\G,\omega)$ a left/right Hamiltonian space for $\G\rra M$, using the left/right action of $\G$ on itself. 	
\end{example}
\begin{example}
	Consider the case of an action groupoid \ref{prop:actiongroupoidtype}. Using the identification 
 $\G\times_M P=G\times P$, the action map $\A$ represents an ordinary $G$-action $\A\colon G\times P\to P$.  Property \eqref{it:b1} is equivalent to the condition that $\omega_P$ is $G$-invariant, with 
	\[ \iota(X_P)\omega_P=-\Phi^* \alpha(X).\]
	This shows, in particular, that $X_P|_p$ with $\alpha(X)|_p=0$ are contained in the kernel of $\omega_P$; 
	Property  \eqref{it:b3} amounts to the condition that this is the entire kernel:
	\[ \ker(\omega_P)|_p=\{X_P|_p\colon \alpha(X)=0\}.\]
\end{example}

\begin{example}
	Let us describe the Hamiltonian spaces for the symplectic groupoid $\G=T^*\R/\ca{K}$ from example \ref{ex:baby2}. Using the quotient map $T^*\R\to \G$, these are in particular Hamiltonian spaces for $T^*\R$: 
	Thus, $P$ is a symplectic manifold with a Hamiltonian $\Phi\in C^\infty(P,\R)$ generating an $\R$-action. For such a space to 
	be a $\G$-space, the subgroupoid $\ca{K}$ must act trivially. This gives the condition on the stabilizer groups, 
	\[ \ca{K}_{\Phi(p)}\subset \G_p\]
	for all $p\in P$. Conversely, if this condition is satisfied, then there is a well-defined $\G$-action given by 
	$[(x,y)]\cdot p=y\cdot p$, so that $(P,\omega_P,\Phi)$ is a Hamiltonian $\G$-space. 
\end{example}
\medskip


\subsection{Morita equivalence of quasi-symplectic groupoids }
	Let $\G_1\rra M_1$ and $\G_2\rra M_2$ be quasi-symplectic groupoids, 
with 2-forms $\omega_i\in \Omega^2(\G_i)$ and twisting 3-forms $\eta_i\in \Omega^3(M_i)$. 
\begin{definition}[Xu \cite{xu:mom}]	

A \emph{Morita equivalence} 
\begin{equation}\label{eq:moritaequivalence2}
\xymatrix{ (\G_1,\omega_1)\ar[d]<2pt>\ar@<-2pt>[d] & (Q,\varpi)\ar[rd]_{\Phi_2}\ar[ld]^{\Phi_1}& (\G_2,\omega_2)\ar[d]<2pt>\ar@<-2pt>[d]\\
	(M_1,\eta_1) & & (M_2,\eta_2)	}
\end{equation}
of quasi-symplectic groupoids is given by a Hilsum-Skandalis bimodule $Q$ of the underlying groupoids, equipped with a 2-form $\varpi\in \Omega^2(Q)$, satisfying the following conditions:
\begin{enumerate}
	\item[(i)] \[ \A^*\varpi=\pr_{\G_1}^*\omega_1+\pr_Q^*\varpi+\pr_{\G_2}^*\omega_2\] 
	where $\A\colon \G_1\times_M Q\times_M \G_2\to Q,\ (g_1,q,g_2)\mapsto g_1\cdot q\cdot g_2$ is the bi-action map. 
	\item[(ii)] \[ \d\varpi=-\Phi_1^*\eta_1+\Phi_2^*\eta_2\]
	\item[(iii)] \[ \ker(\varpi)\cap \ker(T\Phi_1)\cap \ker(T\Phi_2)=0\] 
\end{enumerate}
\end{definition}
Some features: 
\begin{enumerate}
	\item A Morita equivalence of  quasi-symplectic groupoids gives rise to a 1-1 correspondence of their Hamiltonian spaces. In fact, if 
	$(P_2,\omega_{P_2})$ is a Hamiltonian $(\G_2,\omega_2)$-space, then the 2-form $\pr_Q^*\varpi+\pr_{P_2}^*\omega_{P_2}$ on 
	$Q\times_{M_2}P_2$ descends to the quotient $P_1=Q\diamond P_2=(Q\times_{M_2}P_2)/\G_2$, 
		making the latter into a Hamiltonian space
	\[ (P_1,\omega_{P_1})=(Q,\varpi)\diamond (P_2,\omega_{P_2}).\]	
	One may recover $(P_1,\omega_{P_2})$ as $(Q^\op,-\varpi)\diamond (P_1,\omega_{P_1})$. The Hamiltonian spaces share many properties; in particular, there is an identification of symplectic quotients \cite[Corollary 4.20]{xu:mom}
	\item Morita equivalence is an equivalence relation. 
   \item Given a Morita equivalence $(Q,\varpi)$ from $(\G_1,\omega_1)$ to $(\G_2,\omega_2)$, one recovers the groupoids as
   \begin{equation}\label{eq:recovery}
    (\G_1,\omega_1)=(Q,\varpi)\diamond (Q^\op,-\varpi),\ \ \ 
   (\G_2,\omega_2)=(Q^\op,-\varpi)\diamond (Q,\varpi).\end{equation}
\end{enumerate}

	\end{appendix}

\bigskip

\bibliographystyle{amsplain} 

\begin{thebibliography}{10}
	
	\bibitem{al:mom}
	A.~Alekseev, A.~Malkin, and E.~Meinrenken, \emph{{L}ie group valued moment
		maps}, J.~Differential Geom. \textbf{48} (1998), no.~3, 445--495.
	
	\bibitem{al:ati}
	E.~Alekseev, A.~and~Meinrenken, \emph{The {A}tiyah algebroid of the path
		fibration over a {L}ie group}, Lett.~ Math.~ Phys.~ \textbf{90} (2009),
	23--58.
	
	\bibitem{bal:coa}
	J.~Balog, L.~Feh\'{e}r, and L.~Palla, \emph{Coadjoint orbits of the {V}irasoro
		algebra and the global {L}iouville equation}, Internat. J. Modern Phys. A
	\textbf{13} (1998), no.~2, 315--362.
	
	\bibitem{boa:twi}
	P.~Boalch and D.~Yamakawa, \emph{Twisted wild character varieties}, Preprint,
	arXiv:1512.08091.
	
	\bibitem{bur:int}
	H.~Bursztyn, M.~Crainic, A.~Weinstein, and C.~Zhu, \emph{Integration of twisted
		{D}irac brackets}, Duke Math.~J. \textbf{123} (2004), no.~3, 549--607.
	
	\bibitem{cos:gro}
	A.~Coste, P.~Dazord, and A.~Weinstein, \emph{Groupo\"\i des symplectiques},
	Publications du {D}\'epartement de {M}ath\'ematiques. {N}ouvelle {S}\'erie.
	{A}, {V}ol.\ 2, Publ. D\'ep. Math. Nouvelle S\'er. A, vol.~87, Univ.
	Claude-Bernard, Lyon, 1987, pp.~i--ii, 1--62.
	
	\bibitem{dai:coa}
	J.~Dai and D.~Pickrell, \emph{Coadjoint orbits for the central extension of
		{${\rm Diff}^+(S^1)$} and their representatives}, Acta Math. Sci. Ser. B
	(Engl. Ed.) \textbf{24} (2004), no.~2, 185--205.
	
	\bibitem{hoy:lie}
	M.~del Hoyo, \emph{Lie groupoids and their orbispaces}, Port. Math. \textbf{70}
	(2013), no.~2, 161--209.
	
	\bibitem{die:gro}
	T.~Diez and T.~Ratiu, \emph{Group-valued momentum maps for actions of
		automorphism groups}, 2020.
	
	\bibitem{dri:kdv}
	V.~G. Drinfeld and V.~V. Sokolov, \emph{Equations of {K}orteweg-de {V}ries
		type, and simple {L}ie algebras}, Dokl. Akad. Nauk SSSR \textbf{258} (1981),
	no.~1, 11--16.
	
	\bibitem{fre:ver}
	E.~Frenkel and D.~Ben-Zvi, \emph{Vertex algebras and algebraic curves},
	Mathematical Surveys and Monographs, vol.~88, American Mathematical Society,
	Providence, RI, 2001.
	
	\bibitem{ghy:gro}
	E.~Ghys, \emph{Groups acting on the circle}, Enseign. Math. (2) \textbf{47}
	(2001), no.~3-4, 329--407.
	
	\bibitem{gol:the}
	W.~M. Goldman, \emph{{Discontinuous Groups and the Euler Class}}, Ph.D. thesis,
	University of California at Berkeley, 1980.
	
	\bibitem{ham:inv}
	R.~Hamilton, \emph{The inverse function theorem of {N}ash and {M}oser}, Bull.
	Amer. Math. Soc. (N.S.) \textbf{7} (1982), no.~1, 65--222.
	
	\bibitem{igl:dif}
	P.~Iglesias-Zemmour, \emph{Diffeology}, Mathematical Surveys and Monographs,
	vol. 185, American Mathematical Society, Providence, RI, 2013.
	
	\bibitem{khe:inf}
	B.~Khesin and R.~Wendt, \emph{The geometry of infinite-dimensional groups},
	Ergebnisse der Mathematik und ihrer Grenzgebiete. 3. Folge., vol.~51,
	Springer-Verlag, Berlin, 2009.
	
	\bibitem{kir:orb}
	A.~A. Kirillov, \emph{The orbits of the group of diffeomorphisms of the circle,
		and local {L}ie superalgebras}, Funktsional. Anal. i Prilozhen. \textbf{15}
	(1981), no.~2, 75--76.
	
	\bibitem{lau:qua}
	C.~Laurent-Gengoux and P.~Xu, \emph{Quantization of pre-quasi-symplectic
		groupoids and their {H}amiltonian spaces}, The breadth of symplectic and
	{P}oisson geometry, Progr. Math., vol. 232, Birkh\"{a}user Boston, Boston,
	MA, 2005, pp.~423--454.
	
	\bibitem{laz:nor}
	V.~F. Lazutkin and T.~F. Pankratova, \emph{{Normal forms and versal
			deformations for {H}ill's equation}}, Funkcional. Anal. i Prilozen.
	\textbf{9} (1975), no.~4, 41--48.
	
	\bibitem{loi:spi}
	Y.~Loizides, E.~Meinrenken, and Y.~Song, \emph{{Spinor bundles for
			{H}amiltonian loop group spaces}}, J. Symplectic Geom. \textbf{18} (2020),
	no.~3, 889--937.
	
	\bibitem{me:conv}
	E.~Meinrenken, \emph{Convexity for twisted conjugation}, Math. Res. Lett.
	\textbf{24} (2017), no.~6, 1797--1818.
	
	\bibitem{mik:mom}
	K.~Mikami and A.~Weinstein, \emph{Moments and reduction for symplectic
		groupoids}, Publ. Res. Inst. Math. Sci. \textbf{24} (1988), no.~1, 121--140.
	
	\bibitem{moe:fol}
	I.~Moerdijk and J.~Mr{\v{c}}un, \emph{Introduction to foliations and {L}ie
		groupoids}, Cambridge Studies in Advanced Mathematics, vol.~91, Cambridge
	University Press, Cambridge, 2003.
	
	\bibitem{ovs:pro}
	V.~Ovsienko and S.~Tabachnikov, \emph{Projective differential geometry old and
		new}, Cambridge Tracts in Mathematics, vol. 165, Cambridge University Press,
	Cambridge, 2005, From the Schwarzian derivative to the cohomology of
	diffeomorphism groups.
	
	\bibitem{rie:ham}
	A.~Riello and M.~Schiavina, \emph{Hamiltonian gauge theory with corners:
		constraint reduction and flux superselection}, 2022.
	
	\bibitem{sa:jt}
	P.~Saad, S.~Shenker, and D.~Stanford, \emph{Jt gravity as a matrix integral},
	2019.
	
	\bibitem{seg:uni}
	G.~Segal, \emph{Unitary representations of some infinite-dimensional groups},
	Comm. Math. Phys. \textbf{80} (1981), no.~3, 301--342.
	
	\bibitem{seg:geo}
	\bysame, \emph{The geometry of the {K}d{V} equation}, Internat. J. Modern Phys.
	A \textbf{6} (1991), 2859--2869, Topological methods in quantum field theory
	(Trieste, 1990).
	
	\bibitem{sta:jt}
	D.~Stanford and E.~Witten, \emph{J{T} gravity and the ensembles of random
		matrix theory}, Adv. Theor. Math. Phys. \textbf{24} (2020), no.~6,
	1475--1680.
	
	\bibitem{thu:3dim}
	W.~Thurston, \emph{Three-dimensional geometry and topology. {V}ol. 1},
	Princeton Mathematical Series, vol.~35, Princeton University Press,
	Princeton, NJ, 1997, Edited by Silvio Levy.
	
	\bibitem{wit:coa}
	E.~Witten, \emph{Coadjoint orbits of the {V}irasoro group}, Comm. Math. Phys.
	\textbf{114} (1988), no.~1, 1--53.
	
	\bibitem{xu:mom}
	P.~Xu, \emph{Momentum maps and {M}orita equivalence}, J. Differential Geom.
	\textbf{67} (2004), no.~2, 289--333.
	
\end{thebibliography}
\def\cprime{$'$} \def\polhk#1{\setbox0=\hbox{#1}{\ooalign{\hidewidth
			\lower1.5ex\hbox{`}\hidewidth\crcr\unhbox0}}} \def\cprime{$'$}
\def\cprime{$'$} \def\cprime{$'$} \def\cprime{$'$} \def\cprime{$'$}
\def\polhk#1{\setbox0=\hbox{#1}{\ooalign{\hidewidth
			\lower1.5ex\hbox{`}\hidewidth\crcr\unhbox0}}} \def\cprime{$'$}
\def\cprime{$'$} \def\cprime{$'$} \def\cprime{$'$} \def\cprime{$'$}
\providecommand{\bysame}{\leavevmode\hbox to3em{\hrulefill}\thinspace}
\providecommand{\MR}{\relax\ifhmode\unskip\space\fi MR }
\providecommand{\MRhref}[2]{%
	\href{http://www.ams.org/mathscinet-getitem?mr=#1}{#2}
}
\providecommand{\href}[2]{#2}

\end{document}